\documentclass[11pt,a4paper]{amsart}
\input epsf

\usepackage{graphicx, amsmath, amssymb, amscd, amsthm, euscript, amsfonts, xcolor, latexsym, amsbsy}

\usepackage{tikz}
\usetikzlibrary{shapes.geometric, arrows.meta, positioning, calc}
\usetikzlibrary{calc,intersections}

\usepackage{hyperref}
\usepackage{float}
\usepackage{longtable}
\usepackage{subfigure}
\usepackage{graphics, geometry, caption}
\usepackage{makecell}

\newtheorem{thm}{Theorem}[section]
\newtheorem{cor}[thm]{Corollary}
\newtheorem{prop}[thm]{Proposition}
\newtheorem{lem}[thm]{Lemma}
\theoremstyle{definition}
\newtheorem{defn}[thm]{Definition}

\newcommand{\Rmnum}[1]{\uppercase\expandafter{\romannumeral #1}}

\title[On the type of a complex hyperbolic triangle group]{On the type of a complex hyperbolic triangle group}

\author[Wei Liao]{Wei Liao}
\address[Wei Liao]{School of Mathematics, Hunan University, Changsha 410082, P.R. China}
\email{liaowei.math@qq.com}

\author[John R. Parker]{John R. Parker}
\address[John R. Parker]{Department of Mathematical Sciences, Durham University, Durham DH1 3LE, United Kingdom}
\email{j.r.parker@durham.ac.uk}

\date{\today}


\begin{document}

\begin{abstract}
In this paper we consider groups of complex hyperbolic isometries generated by three complex involutions. Our first main result states that if the group is discrete and at least two of the mirrors of complex reflections intersect, then the group may be put into a normal form involving the ordering of traces of short, even length words. Our second main result refines a conjecture of Schwartz (proved by Grossi) that determines the type of the group in terms of the angles between the mirrors. That is, in the parameter space of complex triangles with given angles, it determines which of two given words becomes elliptic first.
\end{abstract}

\maketitle

\section{Introduction}
A \emph{complex hyperbolic triangle group} is a subgroup of $\mathrm{SU}(2,1)$ generated by three distinct involutions $R_j$, $j=1,2,3$, whose mirrors are the complex geodesics $C_j$. If, for every $j=1,2,3$ (taken modulo $3$), the complex geodesics $C_{j-1}$ and $C_{j+1}$ intersect in $\mathbf H_{\mathbb C}^2 \cup \partial\mathbf H_{\mathbb C}^2$ at an angle $\pi/p_j$, where $p_j\in\mathbb N\cup\{\infty\}$ satisfies $p_j\ge2$, then it is called a \emph{complex hyperbolic $(p_1,p_2,p_3)$-triangle group}. When $C_{j-1}$ and $C_{j+1}$ are asymptotic at a point of $\partial{\bf H}^2_{\mathbb C}$ then the angle between them is 0 and we write $p_j=\infty$. If $\min\{p_1,\,p_2,\,p_3\}\ge 3$ and $1/p_1+1/p_2+1/p_3<1$, it is well known that there is a one real parameter family of such groups, and there is a parameter value where the group is discrete and preserves a Lagrangian plane. In a neighbourhood of this parameter value the group remains discrete. Schwartz \cite{sch1} asked about parameter values where discreteness fails. Assuming $3\le p_1\le p_2\le p_3$, Schwartz defined the following pair of words in the $(p_1,p_2,p_3)$-triangle  group $\Gamma(p_1,p_2,p_3)$:
$$
W_A=R_1R_3R_2R_3 \quad \text{and} \quad W_B=R_1R_2R_3.
$$
Schwartz says that $\Gamma(p_1,p_2,p_3)$ has 
\begin{enumerate}
  \item[(a)]   Type-$A$ if there is a parameter value where $W_A$ is parabolic and $W_B$ is loxodromic, and 
\item[(b)] Type-$B$ if there is a parameter value where $W_B$ is parabolic and $W_A$ is loxodromic. 
\end{enumerate}
Moreover, Schwartz conjectured that a triangle group is of Type-$A$ if $p_1<10$, and of Type-$B$ if $p_1>13$. This conjecture was subsequently confirmed by Grossi \cite{gro}. Our Theorem~\ref{main-thm-3} below refines this conjecture using $p_2$ as well as $p_1$.  

In this paper we will focus on groups where $W_A$ is elliptic or parabolic (and $W_B$ may be elliptic, parabolic or loxodromic). It is well known (see \cite{kpt}) that when 
$3\le\min\{p_1,p_2,p_3\}$ the trace of $R_1R_3R_2R_3$ may be used to parametrise the space of $(p_1,p_2,p_3)$-triangle groups. 
Observe that $R_1R_3R_2R_3$ is the product of reflections in $C_1$ and $C_{323}=R_3(C_2)$, 
and so when $R_1R_3R_2R_3$ is elliptic these two mirrors intersect. If the angle between
$C_1$ and $C_{323}$ is $\pi/p_4$ then we call $\langle R_1,R_2,R_3\rangle $
a \emph{complex hyperbolic $(p_1,p_2,p_3;p_4)$-triangle group}, which we denote 
by $\Gamma(p_1,p_2,p_3;p_4)$.

A basic question is whether every complex hyperbolic triangle group has a set of generators
that satisfy the following definition:

\begin{defn}\label{def-good}
A complex hyperbolic triangle group $\Gamma=\langle R_1,R_2,R_3\rangle$ is said to
be \emph{good} if the generators satisfy both of the following conditions:
\begin{enumerate}
\item[(a)] ${\rm tr}(R_2R_3)\le {\rm tr}(R_1R_3)\le {\rm tr}(R_1R_2)\le {\rm tr}(R_1R_3R_2R_3)$;
\item[(b)] if any of ${\rm tr}(R_2R_3)$, ${\rm tr}(R_1R_3)$, ${\rm tr}(R_1R_2)$, and ${\rm tr}(R_1R_3R_2R_3)$ is less than $3$, then it has the form $1+2\cos(2\pi/p)$ for some $p\in{\mathbb N}$ with $p\ge 2$.
\end{enumerate}
\end{defn}

In \cite{kpt}, Kamiya, Parker, and Thompson investigated certain automorphisms of 
complex hyperbolic triangle groups that determine when quadruples $(p_1,p_2,p_3;p_4)$ and
$(p'_1,p'_2,p'_3;p'_4)$ corresponding to the same group, but with different generating 
reflections.
Our first main result builds on the results from \cite{kpt} to show that when 
$\langle R_1,R_2,R_3\rangle$ is discrete and at least one pair of mirrors intersects, then
there is a set of generators $R'_1$, $R'_2$, $R'_3$ satisfying (a) and (b) above.

\begin{thm}\label{main-thm-1}
Any discrete non-elementary complex hyperbolic triangle group where at least one pair of mirrors intersect has a set of generators satisfying Definition~\ref{def-good}.
\end{thm}

A complex hyperbolic triangle group that is not isomorphic to any triangle group with a pair of intersecting mirrors is called a \emph{complex hyperbolic ultra-parallel triangle group}. Such groups were studied by Monaghan, Parker, and Pratoussevitch~\cite{mpp}, and by Liao and Xie~\cite{lx}. In these papers, after relabelling the generators, the authors constructed a large number of ultra-parallel triangle groups and showed that these groups are discrete if and only if the element $R_1R_3(R_2R_3)^n$ is non-elliptic for some specific positive integer $n$.

The necessity follows from the fact that, in an ultra-parallel triangle group, the element $R_1R_3(R_2R_3)^n$ must be nonelliptic. If this were not the case, we could use the involutions
$$
R_1,\qquad R_3(R_2R_3)^n,\qquad R_3(R_2R_3)^{n+1}
$$
as a new generating set. By definition, the resulting triangle group would not be ultra-parallel, yielding a contradiction.

In our second main result we give conditions under which $W_B$ is elliptic, parabolic or loxodromic in a given $(p_1,p_2,p_3;p_4)$-triangle group. In particular, for a $(p_1,p_2,p_3;\infty)$-triangle group $W_A$ is parabolic, and so the group has Type-$A$ if $W_B$ is loxodromic and Type-$B$ if $W_B$ is elliptic. Our result is

\begin{thm}\label{main-thm-2}
Let $p_1,\,p_2,\,p_3,\,p_4 \in{\mathbb Z}$ satisfy $3\le p_1\le p_2\le p_3\le p_4\le \infty$. Define
\begin{equation*}
\begin{aligned}
w &= 1 + 2\cos(2\pi/p_1), \quad & x &= 1 + 2\cos(2\pi/p_2), \\
y &= 1 + 2\cos(2\pi/p_3), \quad & z &= 1 + 2\cos(2\pi/p_4).
\end{aligned}
\end{equation*}
Let $\Gamma(p_1,p_2,p_3;p_4)\subset{\rm SU}(2,1)$ be the corresponding triangle group.
Then:
\begin{enumerate}
\item[(a)] For given $p_1$ and $p_2$ with 
$w+x\ge 50/9$, for all $p_3$ and $p_4$ the element $W_B$
of the group $\Gamma(p_1,p_2,p_3;p_4)$ is elliptic.
\item[(b)] For given $p_1$ and $p_2$ with $w+x\le 50/9$, there exist $p'_3$ and $p'_4$ so that for all $p_3\ge p'_3$ and $p_4\ge p'_4$ the element $W_B$ of the group $\Gamma(p_1,p_2,p_3;p_4)$  is loxodromic. In particular, this holds when
$$
\max\left\{2,\,x,\,\frac{9(w+x)-8}{14}\right\}\le y\le z.
$$
\end{enumerate}
\end{thm}

Moreover, in case (b) and when $p_4=\infty$ we can estimate $p_3$ and so decide
whether the $(p_1,p_2,p_3)$-triangle group $\Gamma(p_1,p_2,p_3)$ has Type-$A$ or Type-$B$.

\begin{thm}\label{main-thm-3}
Let $p_1$, $p_2$ and $p_3$ be integers satisfying $3\le p_1\le p_2\le p_3\le \infty$.
Define
$$
w = 1 + 2\cos(2\pi/p_1), \quad  x = 1 + 2\cos(2\pi/p_2), \quad y = 1 + 2\cos(2\pi/p_3).
$$
Let $\Gamma(p_1,p_2,p_3)\subset{\rm SU}(2,1)$ be the corresponding triangle group.
Then we have 
\begin{enumerate}
\item[(a)] For $p_1$, $p_2$ and $p_3$ so that $w+x\ge 50/9$ then $\Gamma(p_1,p_2,p_3)$ has Type-$B$.
\item[(b)] For $p_1$, $p_2$ and $p_3$ so that $w+x\le 50/9$ and 
$$
\max\left\{1,\,x,\,\frac{3(w+x)}{2}-\frac{16}{3}\right\}\le y,
$$
the group $\Gamma(p_1,p_2,p_3)$ has Type-$A$.
\end{enumerate}
\end{thm}

We remark that when $w+x\le 50/9$ we can see from Theorem~\ref{main-thm-2}~(b) that there
are only finitely many $\Gamma(p_1,p_2,p_3;p_4)$ where $R_1R_2R_3$ is elliptic. We now 
indicate how to find such groups. When $p_1\le 9$ it is easy to list them directly.

\begin{thm}\label{prop-main-finite}
Let $p_1,\,p_2,\,p_3,\,p_4$ be as in Theorem~\ref{main-thm-2}. If $p_1\le 9$ the only groups
$\Gamma(p_1,p_2,p_3;p_4)$ in which $R_1R_2R_3$ is elliptic or parabolic are
\begin{enumerate}
    \item[(a)] $\Gamma(3,3,4;7)$, $\Gamma(3,3,5;5)$, $\Gamma(4,4,4;5)$; 
    \item[(b)] $\Gamma(p,p,p;p)$ for $p\in\{5,6,7,8,9\}$;
    \item[(c)] $\Gamma(8,8,8;9)$, $\Gamma(8,9,9;9)$; 
    \item[(d)] $\Gamma(9,p_2,p_2+1;p_2+1)$ for $9\le p_2\le 16$;
    \item[(e)] $\Gamma(9,p_2,p_2;p_4)$ with $p_2\le p_4\le p_2+\delta(p_2)$ where
     $$
    \delta(p_2) = \begin{cases}
    3, & p_2 \in \{9, 10, 11, 12\}, \\
    2, & p_2 \in \{13, 14, 15, 16, 17\}, \\
    1, & p_2 \in \{18, 19, 20\}, \\
    0, & p_2 \in \{21, 22\}.
    \end{cases}
    $$
\end{enumerate}
\end{thm}

For $10\le p_1\le p_2$ and $w+x\le 50/9$ it is less straightforward to list the groups
$\Gamma(p_1,p_,p_3;p_4)$ where $R_1R_2R_3$ is elliptic. 
However, this can be deduced from Table~\ref{tab:list} where, 
for such $p_1$ and $p_2$, we give the optimal values of 
$p'_3$ and $p'_4$ with $p_2\le p'_3\le p'_4$ so that $R_1R_2R_3$ is loxodromic 
in $\Gamma(p_1,p_2,p_2;p_4)$ for all $p_3\ge p'_3$ and
$p_4\ge p'_4$. (Compare this with Table~2 in Wang \cite{wan}.) The element $R_1R_2R_3$
is elliptic when the group is not on this list. 

In these results we do not address the issue of discreteness. However, there is an extensive
literature on which $\Gamma(p_1,p_2,p_3;p_4)$ are discrete. For example, Thompson \cite{tho} and Deraux \cite{der} show that the groups in 
Theorem~\ref{prop-main-finite}~(a) are lattices. Moreover, the case where
$p_1=p_2=p_3$ was considered by Parker in \cite{par2}. In particular, in Section~4.3 of 
\cite{par2} it is shown that the groups in Theorem~\ref{prop-main-finite}~(b) 
where $p_4=p_1$ are normal subgroups of Livn\'e's lattices \cite{liv}, see also \cite{par1}. 
For groups where $R_1R_2R_3$ 
is loxodromic, Schwartz \cite{sch-rhochi} showed that $\Gamma(4,4,4;7)$ is discrete. 
In our forthcoming paper \cite{lp}, we will show that if $\Gamma(p_1,p_2,p_3;p_4)$ is discrete and both $W_A$ and $W_B$ are elliptic, then it is one of the groups listed in parts~(a) and~(b) of Theorem~\ref{prop-main-finite}.

More is known when $p_4=\infty$. For example, 
$\Gamma(3,3,p_3;\infty)$ is discrete and in ${\rm SU}(2,1)$ for $4\le p_3\le \infty$; see
Deraux-Falbel \cite{der-fal} for $p=4$, Parker, Wang, and Xie \cite{pwx} for 
$4<p_3<\infty$ (see also Acosta \cite{aco}), and Parker and Will 
\cite{pw} for $p_3=\infty$. In \cite{kpt2} and \cite{kam} groups $\Gamma(n,n,\infty;k)$ were 
considered. Changing generators (see Lemma~1.2 of \cite{kpt}), these groups are isomorphic 
to $\Gamma(n,n,k;\infty)$. In \cite{kpt2} and \cite{kam}, it is shown that there is a range of 
values of $k$ where $R_1R_2R_3$ is elliptic, but that it is loxodromic for smaller values of 
$k$. Since the lower bound of this interval is smaller than $n$ this does not contradict our
results. If we were to use the algorithm from Theorem~\ref{main-thm-1} to convert such 
groups to one with a good set of generators, then we would lose the property $p_1=p_2$. 

The classification of complex hyperbolic $(p_1,p_2,p_3)$-triangle groups 
$\Gamma(p_1,p_2,p_3)$ into Type-$A$ and Type-$B$ was solved by Wang in \cite{wan}, 
who showed that whether a triangle group is of Type-$A$ or Type-$B$ is characterised by the 
sign of a polynomial in the cosines of the angles of the triangle; see Theorem~1.1 of \cite{wan}.

To obtain an explicit list of Type-$A$ groups, Wang established a corollary stating that, 
for two complex hyperbolic $(p_1,p_2,p_3)$- and $(p_1',p_2',p_3')$-triangle groups, 
if the former is of Type-$A$ and
$$
p_1'\le p_1,\qquad p_2'\le p_2,\qquad p_3'\ge p_3,
$$
then the latter is also of Type-$A$. By this result, for given $p_1$ and $p_2$, 
once the smallest $p_3'$ for which the polynomial is positive has been found, every 
complex hyperbolic $(p_1,p_2,p_3)$-triangle group with $p_3\ge p_3'$ is of Type-$A$. 
In this way, Wang obtained a list of Type-$A$ groups; see Table~2 of \cite{wan}.

However, \cite{wan} does not appear to provide a general result showing whether, 
for given $p_1$ and $p_2$, such a number $p_3'$ exists. For example, Table~2 of 
\cite{wan} shows that, when $p_1=10$ and $10\le p_2\le 25$, such a number $p_3'$ 
exists. However, no evidence is given as to whether such a number $p_3'$ exists when 
$p_2>25$. As a corollary of statement~(a) of Theorem~\ref{main-thm-3}, we confirm 
that Table~2 of \cite{wan} indeed lists all Type-$A$ groups, since if $R_1R_2R_3$ 
is elliptic in the group $\Gamma(p_1,p_2,p_3;\infty)$, then the corresponding 
triangle group must be of Type-$B$.

The paper is arranged as follows. In Section~\ref{background}, we present the necessary background on complex hyperbolic triangle groups. In Section~\ref{sec-good-triangle} 
we give an algorithm that either produces a good set of generators for 
$\Gamma(p_1,p_2,p_3;p_4)$ or else says that it is not discrete. This proves 
Theorem~\ref{main-thm-1}. In Section~\ref{sec-large-p1-p2} we prove
Theorem~\ref{main-thm-2}~(a) and Theorem~\ref{main-thm-3}~(a). Finally, in
Section~\ref{sec-small-p1-p2} we prove Theorem~\ref{main-thm-2}~(b) and 
Theorem~\ref{main-thm-3}~(b) as well as Theorem~\ref{prop-main-finite}. 

\medskip

\noindent{\bf Acknowledgments.}
This paper was written while the first author was visiting Durham University. We would like to thank the China Scholarship Council for their financial support (Grant No. 202506130039).

\medskip

\section{Background}\label{background}
In this section, we provide some basic background on complex hyperbolic geometry and refer the reader to \cite{gol}.

\subsection{Complex hyperbolic space}
Let $\mathbb C^{2,1}$ denote the $3$-dimensional complex vector space equipped with a  Hermitian form of signature $(2,1)$ defined by
$$
\langle \mathbf z, \mathbf w \rangle
={\bf w}^*H{\bf z}
$$
where ${\bf w}^*$ denotes the Hermitian transpose of ${\bf w}$ (that is transpose ${\bf w}$ 
and complex conjugate its entries). 
Then, the set $\mathbb C^{2,1}\setminus\{0\}$ can be partitioned into the following three subsets:
\begin{eqnarray*}
    V_-&=&\left\{\mathbf z\in \mathbb C^{2, 1}\setminus \left\{0\right\}: \langle \mathbf{z}, \mathbf{z} \rangle<0\right\},\\
    V_0&=&\left\{\mathbf z\in \mathbb C^{2, 1}\setminus \left\{0\right\}: \langle \mathbf{z}, \mathbf{z} \rangle=0\right\},\\
    V_+&=&\left\{\mathbf z\in \mathbb C^{2, 1}\setminus \left\{0\right\}: \langle \mathbf{z}, \mathbf{z} \rangle>0\right\}.
\end{eqnarray*}
Complex hyperbolic space and its boundary are defined by
$$
\mathbf H_{\mathbb C}^2=\mathbb P(V_-)
\qquad\text{and}\qquad
\partial\mathbf H_{\mathbb C}^2=\mathbb P(V_0),
$$
respectively, where $\mathbb P$ is the \emph{standard projection map} 
${\mathbb P}:{\mathbb C}^{2,1}\longmapsto {\mathbb C}{\mathbb P}^2$ given in the chart where $z_3\neq 0$ by
$$
{\mathbb P} : \left(\begin{matrix}
    z_1\\
    z_2\\
    z_3
\end{matrix}\right)\longmapsto
\left(\begin{matrix}
    z_1/z_3 \\
    z_2/z_3
\end{matrix}\right).
$$
Complex hyperbolic space $\mathbf H_{\mathbb C}^2$ is equipped with the \emph{Bergman metric}, which is defined by the distance function $\rho(z,w)$ satisfying
$$
\cosh^2\left(\frac{\rho(z,w)}{2}\right)
=
\frac{\langle \mathbf z, \mathbf w\rangle\langle \mathbf w, \mathbf z \rangle}
{\langle \mathbf z, \mathbf z \rangle\langle \mathbf w, \mathbf w \rangle}.
$$

\medskip

\subsection{Complex Geodesics}

A \emph{complex geodesic} $C$ in ${\bf H}^2_{\mathbb C}$ is the standard projection of a complex $2$-dimensional subspace $V$ of $\mathbb C^{2,1}$, which is determined by its \emph{polar vector} $\mathbf n\in V_+$ by
$$
C={\mathbb P}\bigl\{{\bf z}\in V_-\, :\, \langle {\bf z},{\bf n}\rangle =0\bigr\}.
$$

For ${\bf n}_1,\,{\bf n}_2\in V_+$, we denote
$$
N(\mathbf n_1,\mathbf n_2)=\frac{\langle \mathbf n_1,\mathbf n_2\rangle\langle \mathbf n_2,\mathbf n_1\rangle}{\langle \mathbf n_1,\mathbf n_1\rangle\langle \mathbf n_2,\mathbf n_2\rangle}.
$$
There are three possible relative positions of two distinct complex geodesics $C_1$ and $C_2$, with polar vectors $\mathbf n_1\in V_+$ and $\mathbf n_2\in V_+$, respectively (see
Section 3.3.2 of Goldman \cite{gol}:
\begin{itemize}
    \item \emph{Intersecting}, if $C_1$ and $C_2$ intersect in $\mathbf H_{\mathbb C}^2$, characterized by
    $$
    N(\mathbf n_1,\mathbf n_2)<1,
    $$
    in which case the complex angle $\theta$ between them is given by
    $$
    \cos^2(\theta)=N(\mathbf n_1,\mathbf n_2).
    $$
    \item \emph{Asymptotic}, if $C_1$ and $C_2$ intersect on $\partial\mathbf H_{\mathbb C}^2$, characterized by
    $$
    N(\mathbf n_1,\mathbf n_2)=1.
    $$
    In this case, we say the angle between $C_1$ and $C_2$ is zero.
    \item \emph{Ultra-parallel}, if $C_1$ and $C_2$ do not intersect in $\mathbf H_{\mathbb C}^2\cup\partial\mathbf H_{\mathbb C}^2$, characterized by
    $$
    N(\mathbf n_1,\mathbf n_2)>1,
    $$
    in which case the distance between them is given by
    $$
    \cosh^2\left(\frac{\rho(C_1,C_2)}{2}\right)=N(\mathbf n_1,\mathbf n_2).
    $$
\end{itemize}

\medskip

\subsection{Classifying holomorphic isometries}

The holomorphic isometry group of $\mathbf H_{\mathbb C}^2$ is
$$
\mathrm{PU}(2,1)=\mathrm{SU}(2,1)/\{1,\omega,\omega^2\},
$$
where $\omega$ is a primitive cube root of unity. An isometry $A\in\mathrm{SU}(2,1)$ is classified into one of the following three types according to its fixed points:
\begin{itemize}
    \item \emph{Parabolic}, if $A$ has exactly one fixed point on $\partial\mathbf H_{\mathbb C}^2$.
    \item \emph{Loxodromic}, if $A$ has exactly two distinct fixed points on $\partial\mathbf H_{\mathbb C}^2$.
    \item \emph{Elliptic}, if $A$ has at least one fixed point in $\mathbf H_{\mathbb C}^2$.
\end{itemize}
The type of a given isometry can be determined by the following result of Goldman.

\begin{thm}[Theorem 6.2.4 of \cite{gol}]\label{Goldman}
    Let $F:{\mathbb C}\longrightarrow {\mathbb R}$ be given by
    \begin{eqnarray*}
       F(\tau) & = & |\tau|^4-8\mathrm{Re}(\tau^3)+18|\tau|^2-27 \\
       & = & f_1(\tau)^2+4f_2(\tau)^3,
    \end{eqnarray*}
    where $f_1(\tau)=|\tau|^2+12{\rm Re}(\tau)+9$ and $f_2(\tau)=-2{\rm Re}(\tau)-3$.
    
    Suppose that $A\in \mathrm{SU}(2, 1)$, then
    \begin{enumerate}
        \item[(a)] $A$ is loxodromic if and only if $F(\mathrm{tr}(A))>0$;
        
        \item[(b)] $A$ is elliptic with distinct eigenvalues if and only if $F(\mathrm{tr}(A))<0$;

        \item[(c)] $A$ is elliptic with repeated eigenvalues (or the identity) if and only if $F(\mathrm{tr}(A))=0$ and $A$ is diagonalisable.

        \item[(d)] $A$ is parabolic if and only if $F(\mathrm{tr}(A))=0$ and $A$ is not diagonalisable.
        
    \end{enumerate}
\end{thm}

Let $C$ be a complex geodesic with polar vector ${\bf n}\in V_+$. Then consider the linear map given by 
$$
R\,{\bf z}=-\mathbf z+2\frac{\langle \mathbf z,\mathbf n\rangle}{\langle \mathbf n,\mathbf n \rangle}\mathbf n 
= \left(-I+\frac{2}{\langle{\bf n},{\bf n}\rangle}\,{\bf n}{\bf n}^*H\right){\bf z}. 
$$
Since $R$ is given in terms of the Hermitian form, it corresponds to a matrix in 
${\rm SU}(2,1)$, which we also call $R$. Also, it is clear that the subspace of 
${\mathbb C}^3$ spanned by ${\bf n}$ is a $1$-eigenspace of $R$ and $R$ has a repeated
eigenvalue $(-1)$ and the $(-1)$-eigenspace projects to the complex line $C$.
Thus ${\rm tr}(R)=-1$ and $F({\rm tr}(R))=0$. We call $R$ a \emph{complex reflection} and we call $C$ the 
\emph{mirror} of $R$.

Let $C_1$ and $C_2$ be two complex geodesics with polar vectors ${\bf n}_1$ and 
${\bf n}_2$ respectively, and let $R_1$ and $R_2$ be complex reflections with mirrors $C_1$ and $C_2$. Consider the product $R_1R_2$. Then we have
$$
{\rm tr}(R_1R_2)=4 N({\bf n}_1,{\bf n}_2)-1.
$$
In particular, if $C_1$ and $C_2$ intersect with angle $\theta$ then
$$
{\rm tr}(R_1R_2)=4\cos^2(\theta)-1=1+2\cos(2\theta).
$$
Note that $R_1R_2$ has eigenvalues $1$, $e^{2i\theta}$ and $e^{-2i\theta}$ and that
its $1$-eigenspace projects to $C_1\cap C_2$. 

\subsection{Groups generated by three involutions}

We will be interested in groups generated by three complex reflections $R_1$, $R_2$,
$R_3$ with mirrors $C_1$, $C_2$, $C_3$ and polar vectors ${\bf n}_1$, ${\bf n}_2$, ${\bf n}_3$. Taking indices modulo $3$, we suppose that 
$C_{j-1}$ and $C_{j+1}$ intersect
with a complex angle $\theta_j=\pi/p_j$. That is, 
$N({\bf n}_{j-1},{\bf n}_{j+1})=\cos^2(\pi/p_j)$ and so 
$$
{\rm tr}(R_{j-1}R_{j+1})=4\cos^2(\pi/p_j)-1=1+2\cos(2\pi/p_j). 
$$
(Or in the case where $C_{j-1}$ and $C_{j+1}$ are asymptotic, then ${\rm tr}(R_{j-1}R_{j+1})=3$ and $p_j=\infty$.)
Groups of this type have been widely studied. In particular, see Kamiya, Parker and Thompson
\cite{kpt} where the following normalisation and notation are used
$$
\langle{\bf n}_j,{\bf n}_j\rangle=2,\quad 
\langle{\bf n}_2,{\bf n}_1\rangle=\rho,\quad 
\langle{\bf n}_3,{\bf n}_2\rangle=\sigma,\quad 
\langle{\bf n}_1,{\bf n}_3\rangle=\tau. 
$$
Unlike in the case of constant curvature, when $\min\{p_1,p_2,p_3\}\ge 3$ the three angles $\theta_j=\pi/p_j$ do not determine the three reflections up to ${\rm SU}(2,1)$ equivalence; there is an extra real parameter. In Section 1.5 of \cite{kpt} it is shown that we can take ${\rm tr}(R_1R_3R_2R_3)$ as this extra parameter. (When $\min\{p_1,p_2,p_3\}=2$ then the group is rigid.)
In this paper, we will restrict our attention to the case where the mirrors $C_1$ and $C_{323}=R_3(C_2)$ intersect with angle $\theta_4=\pi/p_4$. We can then relate the four traces to the parameters from Kamiya, Parker, and Thompson \cite{kpt} as follows
\begin{eqnarray*}
{\rm tr}(R_2R_3) & = & 1+2\cos(2\pi/p_1) \, =\, |\sigma|^2-1,\\
{\rm tr}(R_1R_3) & = & 1+2\cos(2\pi/p_2)\, =\, |\tau|^2-1,\\
{\rm tr}(R_1R_2) & = & 1+2\cos(2\pi/p_3)\, =\, |\rho|^2-1,\\
{\rm tr}(R_1R_3R_2R_3) & = & 1+2\cos(2\pi/p_4)\, =\, |\sigma\tau-\bar\rho|^2-1.   
\end{eqnarray*} 
For $p_1$, $p_2$, $p_3$, $p_4$ we let $\Gamma(p_1,p_2,p_3;p_4)$ 
denote the group with these parameters generated by $R_1$, $R_2$ and $R_3$. In Section~\ref{sec-good-triangle} we will show that, after possibly changing the generators for such a group, we may assume $p_1\le p_2\le p_3\le p_4$. In fact our result applies in greater generality, and we show we may assume that if at least one pair of mirrors intersect, then we may take
$$
{\rm tr}(R_2R_3)\le{\rm tr}(R_1R_3)\le{\rm tr}(R_1R_2)\le{\rm tr}(R_1R_3R_2R_3)
$$
and if any of these traces is less than $3$ it has the form above.

We want to investigate when $R_1R_2R_3$ is loxodromic, parabolic or elliptic. In order to do so, we use the Goldman discriminant function $F$ from Theorem~\ref{Goldman}. Again, using Section 1.5 of \cite{kpt} we have
$$
{\rm tr}(R_1R_2R_3)=\rho\sigma\tau-|\rho|^2-|\sigma|^2-|\tau|^2+3.
$$
We can then express the functions $f_1$ and $f_2$ from Theorem~\ref{Goldman}, and therefore $F=f_1^2+4f_2^3$ as polynomials in the four traces, see also Yuhan Wang \cite{wan}. 

\begin{lem}\label{lem-f1-f2}
Let $R_1$, $R_2$, $R_3$ be complex involutions in ${\rm SU}(2,1)$. Write
$$
\begin{array}{ll}
w = {\rm tr}(R_2R_3),\quad &
x = {\rm tr}(R_1R_3), \\
y = {\rm tr}(R_1R_2),\quad &
z = {\rm tr}(R_1R_3R_2R_3).
\end{array}
$$
Then
\begin{eqnarray*}
    f_1(w,x,y,z) & = & \bigl\vert \mathrm{tr}(R_1R_2R_3) \bigr\vert^2
+12\,\mathrm{Re}\bigl(\mathrm{tr}(R_1R_2R_3)\bigr) + 9 \\
& = & (w+x+y-6)(w+x+z-6) \\
&& \quad +6(w+x)-(w+x)^2+wx(7-w-x)-20, \\
f_2(w,x,y,z) & = & 
-2\,\mathrm{Re}\bigl(\mathrm{tr}(R_1R_2R_3)\bigr) - 3 \\
& = & -wx+w+x+y+z-4.
\end{eqnarray*}
\end{lem}

\begin{proof}
We have
\begin{eqnarray*}
\lefteqn{
    \bigl|{\rm tr}(R_1R_2R_3)\bigr| } \\
    & = & |\rho\sigma\tau|^2-2{\rm Re}(\rho\sigma\tau)\bigl(|\rho|^2+|\sigma|^2+|\tau|^2-3\bigr) \\
    && \quad +\bigl(|\rho|^2+|\sigma|^2+|\tau|^2-3\bigr)^2 \\
    & = & (w+1)(x+1)(y+1)+(z-y-(w+1)(x+1))(w+x+y) +1 \\
    && \quad+(w+x+y)^2 \\
    & = & wx+(w+x)(y+z)+yz-wx(w+x)+1 .
\end{eqnarray*}
Similarly
\begin{eqnarray*}
    2{\rm Re}\bigl({\rm tr}(R_1R_2R_3)\bigr)
    & = & 2{\rm Re}(\rho\sigma\tau)-2|\rho|^2-2|\sigma|^2-2|\tau|^2+6 \\
    & = & -z+y+(w+1)(x+1)-2(w+x+y) \\
    & = & wx-w-x-y-z+1. 
\end{eqnarray*}
The result follows by substituting these expressions into $f_1$ and $f_2$. 
\end{proof}

\medskip

Finally, again following Kamiya, Parker, and Thompson \cite{kpt} Section~1.5, we can also express the determinant of the Hermitian form $H$ in terms of $f_2$.

\begin{lem}\cite{kpt}\label{lem-H-sig}
    The Hermitian form $H$ has the desired signature $(2,1)$ if and only if 
    $$
    f_2(w,x,y,z)=-wx+w+x+y+z-4>-1.
    $$
\end{lem}

\section{Making a triangle group good}\label{sec-good-triangle}

In this section we consider more general discrete groups $\Gamma$ generated by three complex 
involutions for which at least one pair of mirrors intersect. 
We will show that, after possibly changing generators, we may assume that
$\Gamma$ is good, that is it satisfies Definition~\ref{def-good}. In other words, 
we can write $\Gamma=\langle R_1,R_2,R_3\rangle$ where the $R_j$ are complex reflections fixing
complex lines so that
$$
{\rm tr}(R_2R_3)\le{\rm tr}(R_1R_3)\le{\rm tr}(R_1R_2)\le{\rm tr}(R_1R_3R_2R_3)
$$
and if any of them is less than $3$ then it has the form $1+2\cos(2\pi/p)$ for some $p\in{\mathbb Z}$. 
In particular, for $(p_1,p_2,p_3;p_4)$ triangle groups we may change generators so that
$p_1\le p_2\le p_3\le p_4$. 

\begin{prop}\label{prop-good}
Let $R_1$, $R_2$, $R_3$ be involutions in ${\rm SU}(2,1)$ fixing complex geodesics $C_1$, $C_2$, $C_3$ in 
${\bf H}^2_{\mathbb C}$, respectively. Let $\Gamma=\langle R_1, R_2, R_3\rangle$ be the group they generate. 
Suppose that $\Gamma$ is 
discrete and non-elementary and that at least one pair of mirrorsof group elements intersect
(so $\Gamma$ contains at least one elliptic map that can be written as the product of two complex reflections). 
Then there exist an ordered set of involutions $\{R'_1,\,R'_2,\,R'_3\}$ generating $\Gamma$ 
so that
\begin{enumerate}
\item[(a)] ${\rm tr}(R'_2R'_3)\le {\rm tr}(R'_1R'_3)\le {\rm tr}(R'_1R'_2)\le {\rm tr}(R'_1R'_3R'_2R'_3)$;
\item[(b)] if any of ${\rm tr}(R'_2R'_3)$, ${\rm tr}(R'_1R'_3)$, ${\rm tr}(R'_1R'_2)$, ${\rm tr}(R'_1R'_3R'_2R'_3)$
is less than $3$ then it has the form $1+2\cos(2\pi/p)$ for some $p\ge 2\in{\mathbb N}$.
\end{enumerate}
\end{prop}

The following corollary follows directly.

\begin{cor}\label{cor-good}
Let $\Gamma$ and $R'_1$, $R'_2$, $R'_3$ be as in Proposition~\ref{prop-good}. If ${\rm tr}(R'_1R'_3R'_2R'_3)<3$ then
$$
\begin{array}{ll}
{\rm tr}(R'_2R'_3)=1+2\cos(2\pi/p_1),\quad &
{\rm tr}(R'_1R'_3)=1+2\cos(2\pi/p_2), \\
{\rm tr}(R'_1R'_2)=1+2\cos(2\pi/p_3),\quad &
{\rm tr}(R'_1R'_3R'_2R'_3)=1+2\cos(2\pi/p_4),
\end{array}
$$
where $p_1,\,p_2,\,p_3,\,p_4\in{\mathbb Z}$ satisfy $2\le p_1\le p_2\le p_3\le p_4$.
\end{cor}

In order to prove Proposition~\ref{prop-good}, we give an algorithm that constructs $R'_1$, $R'_2$, $R'_3$. 
In what follows, we associate an ordered quadruple of positive real numbers $(w,x,y,z)$ to the ordered set 
$\{R_1,\,R_2,\,R_3\}$. These are defined by
$$
w={\rm tr}(R_2R_3),\quad x={\rm tr}(R_1R_3),\quad y={\rm tr}(R_1R_2),\quad z={\rm tr}(R_1R_3R_2R_3).
$$
When the $R_j$ are given a superscript etc then we give $w$, $x$, $y$, $z$ the same superscript.

We define two processes, called {\sl reorder} and {\sl reduce} that give steps in the algorithm.

\textbf{(1) The process reorder.} 
This process uses two automorphisms $\iota_3$ and $\iota_{23}$ defined as follows:
\begin{enumerate}
\item[(a)] The automorphism $\iota_3$ (see Section~1.4 of Kamiya, Parker, Thompson \cite{kpt}):
$$
\iota_3(R_1)=R_3R_1R_3,\quad \iota_3(R_2)=R_2,\quad \iota_3(R_3)=R_3.
$$
Then $\iota_3$ has the following effect on the traces (see Lemma~1.2 of \cite{kpt}):
$$
(w,x,y,z) \longmapsto (w,x,z,y).
$$
\item[(b)] The automorphism $\iota_{23}$:
$$
\iota_{23}(R_1)=R_1,\quad \iota_{23}(R_2)=R_3,\quad \iota_{23}(R_3)=R_2.
$$
Then $\iota_{23}$ has the following effect on the traces:
$$
(w,x,y,z)\longmapsto \bigl(w,y,x,z+w(y-x)\bigr).
$$
Note that $\iota_{23}(R_1R_3R_2R_3)=R_1R_2R_3R_2$. Using the formulae in Section~1.5 of Kamiya, Parker, and Thompson
\cite{kpt} we find that
\begin{eqnarray*}
{\rm tr}(R_1R_2R_3R_2) & = & |\rho\sigma-\bar{\tau}|^2-1 \\
& = & |\sigma\tau-\bar\rho|^2-1+(|\sigma|^2-1)(|\rho|^2-|\tau|^2) \\
& = & z+w(y-x).
\end{eqnarray*}
\end{enumerate}

The process reorder is then:
\begin{enumerate}
\item[(a)] If $z<y$ then perform $\iota_3$.
\item[(b)] If $y<x$ then perform $\iota_{23}$, and note that this means  $z+w(y-x)<z$.
If $z+w(y-x)<3$ and does not have the form $1+2\cos(2k\pi/p)$ for $k/p\in{\mathbb Q}$ then 
$\Gamma$ is not discrete and we stop. Otherwise proceed. 
\item[(c)] If $x\le y\le z$ then stop. 
\end{enumerate}

Note that each new trace produced by this process is bounded above by $\max\{x,y,z\}$.
We observe that the process reorder terminates after finitely many steps. In order to see this, we note that 
$\iota_{23}$ and $\iota_3$ are both involutions and so they generate a dihedral group. Any element of this group
either preserves or swaps $R_2$ and $R_3$ and it conjugates $R_1$ by an element of $\langle R_2,R_3\rangle$. Since we are interested in the case where $R_2R_3$ is elliptic of finite order, this process only generates finitely many new traces. So after finitely many steps
it either stops because it has found a new set of traces with $w\le x\le y$ as in (c), or else
it stops because it has produced an elliptic map if infinite order, as in (b). 

We note that
if $R_2R_3$ is parabolic or loxodromic, then it is not hard to show that there are only finitely many $m\in{\mathbb Z}$ so that $\bigl|{\rm tr}\bigl(R_1R_3(R_2R_3)^m\bigr)\bigr|$ is uniformly bounded above. Therefore, reorder stops after finitely many steps in these cases as well.

\medskip
 
\textbf{(2) The process reduce.}
We suppose that $y={\rm tr}(R_1R_3)=1+2\cos(2k\pi/p)$ where $2\le k\le p/2$ is coprime to $p$. 
(That is $R_1R_3$ has eigenvalues $1$, $e^{2k\pi i/p}$ and $e^{-2k\pi i/p}$.)
This means
there exist $r$ and $s$ so that $1=kr+ps$. Define $\iota$ by:
$$
\iota(R_1) = (R_3R_1)^rR_3,\quad \iota(R_2)=R_2,\quad \iota(R_3)=R_3.
$$
Since $(R_1R_3)$ has order $p$ we have $R_1R_3=\bigl((R_1R_3)^r\bigr)^k$ and so $\iota$ is indeed an automorphism.
We note that 
$$
{\rm tr}\bigl((R_1R_3)^r\bigr)=1+2\cos(2kr\pi/p)=1+2\cos\bigl(2(1-ps)\pi/p\bigr)=1+2\cos(2\pi/p).
$$

We can now prove Proposition~\ref{prop-good}.

\begin{proof} (Proposition~\ref{prop-good}.)
Since $\Gamma$ contains at least one pair of reflections whose mirrors intersect, and so
their product is elliptic, 
we may assume without loss of generality that $R_2R_3$ is elliptic and its order
$p_1\ge 2$ is the smallest order of any 
elliptic element of $\Gamma$ that can be written as the product
of two reflections. Let $C_2$ and $C_3$ be the mirrors of $R_2$ and $R_3$, and let
$v_1\in{\bf H}^2_{\mathbb C}$ be their intersection point. 
By taking powers if necessary, we also suppose the angle between $C_2$ and $C_3$
is $\pi/p_1$, and so ${\rm tr}(R_2R_3)=1+2\cos(2\pi/p_1)$.

We can now define the algorithm. 
\begin{enumerate}
\item[(1)] Write $R_1^{(1)}=R_1$, $R_2^{(1)}=R_2$, $R_3^{(1)}=R_3$. 
\item[(2)] Given $\{R_1^{(n)},\,R_2^{(n)},\,R_2^{(n)}\}$ perform the process reorder. 
If this does not stop because of an elliptic map of infinite order, then it produces 
new reflections $\{\widetilde{R}_1^{(n)},\,\widetilde{R}_2^{(n)},\,\widetilde{R}_3^{(n)}\}$.
Note that 
$$
{\rm tr}\bigl(\widetilde{R}_2^{(n)}\widetilde{R}_3^{(n)}\bigr)={\rm tr}(R_2R_3)=w.
$$
Define
$$
\tilde{x}_n={\rm tr}\bigl(\widetilde{R}_1^{(n)}\widetilde{R}_3^{(n)}\bigr),\quad
\tilde{y}_n={\rm tr}\bigl(\widetilde{R}_1^{(n)}\widetilde{R}_2^{(n)}\bigr),\quad 
\tilde{z}_n={\rm tr}\bigl(\widetilde{R}_1^{(n)}\widetilde{R}_3^{(n)}\widetilde{R}_2^{(n)}\widetilde{R}_3^{(n)}\bigr).
$$
Note that $w\le \tilde{x}_n\le \tilde{y}_n\le \tilde{z}_n$.
\item[(3a)] Given $\{\widetilde{R}_1^{(n)},\,\widetilde{R}_2^{(n)},\,\widetilde{R}_3^{(n)}\}$, if 
$$
\tilde{x}_n=1+2\cos(2k\pi/\tilde{p}_2^{(n)}),
$$ 
where $2\le k \le \tilde{p}_2^{(n)}/2$
is coprime to $\tilde{p}_2^{(n)}$, then perform reduce. Denote the resulting ordered triple of reflections by
$\{R_1^{(n+1)},\,R_2^{(n+1)},\,R_3^{(n+1)}\}$.
Note that 
$$
x_{n+1}={\rm tr}(R_1^{(n+1)}R_3^{(n+1)})=1+2\cos(2\pi/\tilde{p}_2).
$$
\item[(3b)] Given $\{\widetilde{R}_1^{(n)},\,\widetilde{R}_2^{(n)},\,\widetilde{R}_3^{(n)}\}$, if 
$$
\tilde{x}_n =1+2\cos(2\pi/\tilde{p}_2^{(n)}),\quad 
\tilde{y}_n=1+2\cos(2k\pi/\tilde{p}_3^{(n)}),
$$ 
where $2 \le k \le \tilde{p}_3^{(n)}/2$
is coprime to $\tilde{p}_3^{(n)}$, then apply $\iota_{23}$ and then perform reduce.
Denote the resulting ordered triple of reflections by
$\{R_1^{(n+1)},\,R_2^{(n+1)},\,R_3^{(n+1)}\}$. Note that 
$$
x_{n+1}={\rm tr}(R_1^{(n+1)}R_3^{(n+1)})=1+2\cos(2\pi/\tilde{p}_3).
$$
\item[(3c)] Given $\{\widetilde{R}_1^{(n)},\,\widetilde{R}_2^{(n)},\,\widetilde{R}_3^{(n)}\}$, if 
$$
\begin{array}{ll}
\tilde{x}_n =1+2\cos(2\pi/\tilde{p}_2^{(n)}),\quad &
\tilde{y}_n =1+2\cos(2\pi/\tilde{p}_3^{(n)}), \\
\tilde{z}_n=1+2\cos(2k\pi/\tilde{p}_4^{(n)}), & 
\end{array}
$$ 
where $2\le k\le \widetilde{p}_4^{(n)}/2$ is coprime to $\tilde{p}_4^{(n)}$, then apply $\iota_3$ followed by $\iota_{23}$ and then perform reduce.
Denote the resulting ordered triple of reflections by
$\{R_1^{(n+1)},\,R_2^{(n+1)},\,R_3^{(n+1)}\}$. Note that 
$$
x_{n+1}={\rm tr}\bigl(R_1^{(n+1)}R_3^{(n+1)}\bigr)=1+2\cos(2\pi/\tilde{p}_4).
$$
\item[(3d)] Given $\{\widetilde{R}_1^{(n)},\,\widetilde{R}_2^{(n)},\,\widetilde{R}_3^{(n)}\}$, if 
$$
\begin{array}{ll}
\tilde{x}_n=1+2\cos(2\pi/\tilde{p}_2^{(n)}), & \hbox{ or } \quad \tilde{x}_n\ge 3, \\
\tilde{y}_n=1+2\cos(2\pi/\tilde{p}_3^{(n)}), & \hbox{ or } \quad \tilde{y}_n\ge 3, \\
\tilde{z}_n=1+2\cos(2\pi/\tilde{p}_4^{(n)}), & \hbox{ or } \quad \tilde{z}_n\ge 3,
\end{array}
$$
then stop.
\end{enumerate}
It remains to show that this algorithm terminates after finitely many steps.

Write $C_j^{(n)}$ for the mirror of $R_j^{(n)}$ and $\widetilde{C}_j^{(n)}$ for the mirror of $\widetilde{R}_j^{(n)}$. 
Observe that, at each stage, $R_2^{(n)}$ is one of $R_2$ or $R_3$ and $R_3^{(n)}$ is the other one. Similarly
for $\widetilde{R}_2^{(n)}$ and $\widetilde{R}_3^{(n)}$. This means that for all $n$ 
$$
C_2^{(n)}\cap C_3^{(n)}=\widetilde{C}_2^{(n)}\cap \widetilde{C}_3^{(n)}=C_2\cap C_3=v_1.
$$
We claim that for all $n$
\begin{equation}\label{eq-claim-good}
\rho\bigl(v_1,\widetilde{C}_1^{(n)}\bigr)=\rho\bigl(v_1,C_1^{(n)}\bigr),\qquad 
\rho\bigl(v_1,C_1^{(n+1)}\bigr)<\rho\bigl(v_1,\widetilde{C}_1^{(n)}\bigr).
\end{equation}
First consider the process reorder. 
If we apply $\iota_{23}$ then $\iota_{23}(R_1)=R_1$, so its mirror is $C_1$ and $\rho(v_1,C_1)$ remains unchanged.
If we apply $\iota_3$ then $\iota_3(R_1)=R_3R_1R_3$ whose mirror is $R_3(C_1)$. We have
$$
\rho\bigl(v_1,R_3(C_1)\bigr)=\rho\bigl(R_3(v_1),R_3(C_1)\bigr)=\rho(v_1,C_1).
$$
For the first equality we used the fact that $R_3$ fixes $v_1$ and for the second equality we use that $R_3$ is an isometry. 
Therefore, when we apply the process reorder the distance between $v_1$ and $C_1$ does not change. This proves
the first identity from \eqref{eq-claim-good}
$$
\rho\bigl(v_1,\widetilde{C}_1^{(n)}\bigr)=\rho\bigl(v_1,C_1^{(n)}\bigr).
$$

Now we consider the process reduce. 
Let 
$$
\tilde{v}_2^{(n)}=\tilde{C}_3^{(n)}\cap \tilde{C}_1^{(n)}
\quad \hbox{ and }\quad v_2^{(n+1)}=C_3^{(n+1)}\cap C_1^{(n+1)}
$$ 
be the fixed
points of $\tilde{R}_1^{(n)}\tilde{R}_3^{(n)}$ and $R_1^{(n+1)}R_3^{(n+1)}$. By construction, 
$R_1^{(n+1)}R_3^{(n+1)}$ is a power of $\tilde{R}_1^{(n)}\tilde{R}_3^{(n)}$ and so their fixed points are the same. 
That is $\tilde{v}_2^{(n)}=v_2^{(n+1)}$ and we choose to simply write it as $\tilde{v}_2^{(n)}$.

Also, let $\tilde{u}_n$ be the foot of the perpendicular from $v_1$ to $\widetilde{C}_1^{(n)}$ and 
$u_{n+1}$ be the foot of the perpendicular from $v_1$ to $C_{n+1}$. Note that
$$
\rho\bigl(v_1,\widetilde{C}_1^{(n)}\bigr)=\rho(v_1,\tilde{u}_n),\qquad
\rho\bigl(v_1,C_1^{(n+1)}\bigr)=\rho(v_1,u_{n+1}).
$$
We now use the sine rule (see Theorem 3.2.1 of Goldman \cite{gol}) on the
complex triangle with vertices $v_1$, $\tilde{v}_2^{(n)}$ and $\tilde{u}_n$. This triangle has internal complex angles 
$k\pi/\tilde{p}_2^{(n)}$ at $\tilde{v}_2^{(n)}$ and $\pi/2$ at $\tilde{u}_n$. We do the same with the 
complex triangle with vertices $v_1$, $\tilde{v}_2^{(n)}$ and $u_{n+1}$. This has internal complex angles 
of $\pi/\tilde{p}_2^{(n)}$ at $\tilde{v}_2^{(n)}$ and $\pi/2$ at $u_{n+1}$. Therefore
\begin{eqnarray*}
\sinh\left(\frac{\rho(v_1,C_1^{(n+1)})}{2}\right) 
& = & \sinh\left(\frac{\rho(v_1,u_{n+1})}{2}\right) \\
& = & \sin\left(\frac{\pi}{\tilde{p}_2^{(n)}}\right)\sinh\left(\frac{\rho(v_1,\tilde{v}_2^{(n)})}{2}\right) \\
& < & \sin\left(\frac{k\pi}{\tilde{p}_2^{(n)}}\right)\sinh\left(\frac{\rho(v_1,\tilde{v}_2^{(n)})}{2}\right) \\
& = & \sinh\left(\frac{\rho(v_1,\tilde{u}_n)}{2}\right) \\
& = & \sinh\left(\frac{\rho( v_1,\tilde{C}_1^{(n)})}{2}\right).
\end{eqnarray*}
This completes the proof of \eqref{eq-claim-good}.

We are now in a position to prove that the algorithm terminates after finitely many steps. Suppose for a contradiction
that the algorithm does not terminate. Then it will give an infinite family $\{R_1^{(n)}\}$ of reflections in $\Gamma$ with
mirrors $\{C_1^{(n)}\}$. Let $K$ be the closed ball with centre $v_1$ and radius $r=\rho(v_1,C_1)$. 
From \eqref{eq-claim-good} we have 
$$
0<\ldots<\rho\bigl(v_1,C_1^{(n+1)}\bigr)<\rho\bigl(v_1,C_1^{(n)}\bigr)<\cdots<\rho(v_1,C_1^{(1)})=\rho(v_1,C_1)=r.
$$
Therefore the complex geodesics $\{C_1^{(n)}\}$ are distinct and they all intersect $K$. Since $R_1^{(n)}$ fixes 
$C_1^{(n)}$ this means that $R_1^{(n)}(K)\cap K$ is nonempty for all $n$. Therefore $\Gamma$ does not act
properly discontinuously on ${\bf H}^2_{\mathbb C}$, which is a contradiction to our hypothesis that $\Gamma$ is
discrete.

Therefore, after finitely many steps we obtain an ordered triple $\{\tilde{R}_1^{(n)},\,\tilde{R}_2^{(n)},\,\tilde{R}_3^{(n)}\}$
so that $\tilde{x}_n\le\tilde{y}_n\le\tilde{z}_n$ and 
$$
\begin{array}{ll}
\tilde{x}_n=1+2\cos(2\pi/\tilde{p}_2^{(n)}), & \hbox{ or } \quad \tilde{x}_n\ge 3, \\
\tilde{y}_n=1+2\cos(2\pi/\tilde{p}_3^{(n)}), & \hbox{ or } \quad \tilde{y}_n\ge 3, \\
\tilde{z}_n=1+2\cos(2\pi/\tilde{p}_4^{(n)}), & \hbox{ or } \quad \tilde{z}_n\ge 3.
\end{array}
$$
We also have 
$$
w={\rm tr}(R_2R_3)={\rm tr}\bigl(\tilde{R}_2^{(n)}\tilde{R}_3^{(n)}\bigr)=1+2\cos(2\pi/p_1).
$$
By hypothesis, $p_1$ was chosen to be the smallest order of an elliptic element of $\Gamma$. This means that
$p_1\le \tilde{p}_2^{(n)}$ and so $w\le \tilde{x}_n$. This proves the result.

\end{proof}

\medskip

We remark that if we drop the condition that at least one pair of mirrors intersect (which
guarantees the existence of the point $v_1$) then it is not immediately clear that the algorithm will stop after finitely many steps.

\section{The main result for large values of $p_1$ and $p_2$}\label{sec-large-p1-p2}

In this section we consider groups $\Gamma(p_1,p_2,p_3;p_4)$ with 
$3\le p_1\le p_2\le p_3\le p_4\le\infty$ where 
$$
50/9\le w+x={\rm tr}(R_2R_3)+{\rm tr}(r_1R_3)=\bigl(2\cos(2\pi/p_1)+1\bigr)
+\bigl(2\cos(2\pi/p_2)+1\bigr).
$$
The main result of this section is:

\begin{prop}\label{prop-main-large}
    Let $p_1$, $p_2$, $p_3$, $p_4$ be integers satisfying 
    $3\le p_1\le p_2\le p_3\le p_4\le \infty$. Define
    $$
    \begin{array}{ll}
    w=1+2\cos(2\pi/p_1), \quad & 
    x=1+2\cos(2\pi/p_2), \quad \\
    y=1+2\cos(2\pi/p_3), \quad & 
    z=1+2\cos(2\pi/p_4)+1.
    \end{array}
    $$
    If $p_1$ and $p_2$ satisfy $w+x\ge 50/9$ then  in the group $\Gamma(p_1,p_2,p_3;p_4)$,
    for all $p_3$ and $p_4$, the element $W_B=R_1R_2R_3$ is elliptic.
\end{prop}

Setting $p_4=\infty$, so $R_1R_3R_2R_3$ is parabolic, we immediately obtain

\begin{cor}\label{cor-main-large}
    For integers $p_1$, $p_2$, $p_3$ satisfying $3\le p_1\le p_2\le p_3
    \le \infty$, we define
    $$
    w=1+2\cos(2\pi/p_1), \quad 
    x=1+2\cos(2\pi/p_2), \quad 
    y=1+2\cos(2\pi/p_3).
    $$
    If $w+x\ge 50/9$ then the triangle group $\Gamma(p_1,p_2,p_3)$ has Type-$B$.
\end{cor}

Our strategy for proving this result will be to use the polynomials $f_1(w,x,y,z)$ and
$f_2(w,x,y,z)$ from Lemma~\ref{lem-f1-f2} and to show that when $w+x\ge 50/9$ then
$$
F(w,x,y,z)=f_1(w,x,y,z)^2+4f_2(w,x,y,z)^3<0. 
$$

For all the results below, with the exception of Lemma~\ref{lemma-11}, we will suppose that $w$, $x$, $y$, $z$ are continuous variables lying in the interval $[0,3]$. In particular, we will take partial derivatives of $F(w,x,y,z)$ 
with respect to various functions of these variables.

It will often be useful to make a simple change of variables. Namely, we write 
$$
u=(w+x)/2\qquad \text{and}\qquad v=\sqrt{wx},
$$
so 
$$
w=u-\sqrt{u^2-v^2},\qquad \text{and}\qquad
x=u+\sqrt{u^2-v^2}.
$$
Write $\tilde{F}(u,v,y,z)=F(w,x,y,z)$ for this function in the new variables. 
Note that we have
$$
w\le v\le u \le x
$$
with equality in each place if and only if $w=x$. Moreover, Note that when
$w\le x\le 3$ we have $u\le (w+3)/2$ and 
\begin{equation}\label{eq-u-v-ineq}
0\le (3-w)(3-x)=9-6u+v^2.
\end{equation}
Thus
$$
u\le 3/2+v^2/6
$$
with equality if and only if $x=3$.

\medskip

It will be useful to give inequalities for $w$, $x$, and $v$ that follow from our hypothesis $w+x\ge 50/9$. 

\begin{lem}\label{lem-6}
Suppose that $w\le x\le 3$. If $w+x=2u\ge 50/9$ then
$w\ge 23/9$ and $wx=v^2\ge 23/3$. In particular, $p_1\ge 10$. 
\end{lem}

\begin{proof}
First, observe that
$$
w=(w+x)-x\ge 50/9-3=23/9.
$$
Note that $1+2\cos(2\pi/9)<23/9<1+2\cos(2\pi/10)$ and so $p_1\ge 10$.
Also using \eqref{eq-u-v-ineq} we have
$$
v^2\ge 6u-9\ge 50/3-9=23/3.
$$
\end{proof}

\medskip

The proof of Proposition~\ref{prop-main-large} will be by way of a series of lemmas, each of which is proved using elementary calculus. In particular, we will keep some variables constant and show that the function $F$ is either monotone or convex in the other variable(s). We do this in two stages.
In the first stage, we eliminate dependence on $y$ and $z$, see Corollary~\ref{cor-1}. We change variables to $u=(w+x)/2$ and $v=\sqrt{wx}$ and we use the hypothesis that $v\ge \sqrt{23/3}$ rather than $2u\ge 50/9$, see Lemma~\ref{lem-6}.

\begin{lem}\label{lemma-1}
Suppose $8/3\le v\le 3$ and $y\le \min\{2v-8/3,\,3\}$.
Then:
\begin{enumerate}
    \item[(a)] $\tilde{F}(u,v,y,z)\le \tilde{F}(u,v,y,y)$;
\item[(b)] $\tilde{F}(u,v,y,y)\le \tilde{F}(u,v,x,x)$, 
with equality if and only if $x=y$.
\end{enumerate}
\end{lem}

Note that $8/3<\sqrt{23/3}$ and so this result applies when $w+x\ge 50/9$.

\begin{lem}\label{lemma-2}
Suppose $\sqrt{23/3}\le v\le 17/6$ and $2v-8/3\le y\le 3$.
Then:
\begin{enumerate}
\item[(a)] $\tilde{F}(u,v,y,z)\le \max\bigl\{\tilde{F}(u,v,y,y),\tilde{F}(u,v,y,3)\bigr\}$;
\item[(b)] $\tilde{F}(u,v,y,y)\le \max\bigl\{\tilde{F}(u,v,x,x),\tilde{F}(u,v,3,3)\bigr\}$;
\item[(c)] $\tilde{F}(u,v,y,3)\le \max\bigl\{\tilde{F}(u,v,x,3),\,\tilde{F}(u,v,3,3)\bigr\}$.
\end{enumerate}
Note that if $v>17/6$ then $2v-8/3>3$ and the condition on $y$ is vacuous. 
\end{lem}

The following corollary is obtained by combining the results of Lemmas~\ref{lemma-1} and~\ref{lemma-2} and writing the function $F$ in terms of the variables $(w,x,y,z)$.

\begin{cor}\label{cor-1}
Suppose that $w+x\ge 50/9$.
Then
$$
 F(w, x, y, z)\le 
 \max\bigl\{F(w,x,x,x),\,F(w,x,x,3),\,F(w, x, 3, 3)\bigr\}.  
$$
\end{cor}
\medskip

The second stage of the proof of Proposition~\ref{prop-main-large} involves finding conditions under which each of the functions from Corollary~\ref{cor-1}, namely $F(w,x,x,x)$, $F(w,x,x,3)$ and $F(w,x,3,3)$, is non-positive. For these lemmas, we assume $w\ge 23/9$, which is 
true when $w+x\ge 50/9$; see Lemma~\ref{lem-6}.

\medskip

\begin{lem}\label{lemma-8}
Suppose that $23/9\le  w\le 3$. Then
\begin{enumerate}
    \item[(a)] $F(w,x,x,x)\le 0$ with equality if and only if $w=x$ or $(w,x)=(23/9,3)$; 
    \item[(b)] if $\max\bigl\{50/9-w,\,w\bigr\}\le x\le 3$ then $F(w,x,x,3)\le 0$ with equality if and only if $(w,x)=(23/9,3)$ or $(3,3)$;
    \item[(c)] if $\max\bigl\{(5+3\sqrt{33})/4-w,\,w\bigr\}\le x\le 3$ then $F(w, x, 3, 3)\le 0$ with equality if and only if $(w,x)=(23/9,3)$, $(3,3)$ or $\bigl((5+3\sqrt{33})/8,(5+3\sqrt{33})/8\bigr)$. 
\end{enumerate}
\end{lem}

The bound on $w+x=2u$ of $w+x\ge (5+3\sqrt{33})/4$ from Lemma~\ref{lemma-8}~(c) is larger 
than the bound $w+x\ge 50/9$ that we want. However, we will show that there are no
values of $p_1$ and $p_2$ for which 
$w=1+2\cos(2\pi/p_1)$, $x=1+2\cos(2\pi/p_2)$
and $50/9\le (w+x)\le (5+3\sqrt{33})/4$.

\begin{lem}\label{lemma-11}
Let $p_1,\,p_2\in{\mathbb Z}\cup \{\infty\}$ with $3\le p_1\le p_2$. Then either 
$$
\bigl(1+2\cos(2\pi/p_1)\bigr)+\bigl(1+2\cos(2\pi/p_2)\bigr)< 50/9,
$$
or
$$
\bigl(1+2\cos(2\pi/p_1)\bigr)+\bigl(1+2\cos(2\pi/p_2)\bigr)> (5+3\sqrt{33})/4.
$$
\end{lem}

\medskip

We now provide proofs of the lemmas mentioned above.

\medskip

\begin{proof} (Lemma~\ref{lemma-1}.)
We suppose $y\le \min\{3v-8/3,3\}$ and $8/3\le v\le 3$. 

We first prove part (a) of the lemma which says $\tilde{F}(u,v,y,z)\le\tilde{F}(u,v,y,y)$. 
Define $\tilde{g}_1$ by  
$$
\tilde{F}(u,v,y,z)-\tilde{F}(u,v,y,y)=(z-y)\tilde{g}_1(u,v,y,z).
$$ 
That is
\begin{eqnarray*}
\tilde{g}_1(u,v,y,z) & = & 
-8u^2v^2 - 4uv^2y + 12v^4 \\
&& \quad + 12u^2y + 4u^2z + 4uv^2 + 8uy^2 + 4uyz \\
&& \quad - 22v^2y - 12v^2z + y^3 + y^2z \\
&& \quad - 24uy + 12v^2 + 4y^2 + 4yz + 4z^2 + 16u - 4y - 12z.
\end{eqnarray*}

Consider 
\begin{eqnarray*}
\frac{\partial \tilde{g}_1}{\partial u}(u,v,y,z)
= 4(6y + 2z -4v^2)u + 8y^2 + 4yz - 24y + 16+ 4v^2(1-y) .
\end{eqnarray*}
If $2\le y\le z\le 3$ and $v^2\ge 13/2$ then
\begin{eqnarray*}
6y+2z-4v^2 & \le & 6y+2z-26 \le -2,\\
8y^2 + 4yz - 24y + 16+ 4v^2(1-y) & \le & 8y^2+12y-24y+16+26-26y \\
& = & 2(y-3)(4y-7)\le 0.
\end{eqnarray*}
Hence
$$
\frac{\partial \tilde{g}_1}{\partial u}(u,v,y,z)<0
$$
and so 
$$
\tilde{g}_1(u,v,y,z)\le \tilde{g}_1(v,v,y,z).
$$

For fixed $u$ and $v$ consider the Hessian matrix $H_{u,v}(y,z)$ of $\tilde{g}_1(u,v,y,z)$ with respect to the 
variables $y$ and $z$. It is given by
$$
H_{u,v}(y,z) = \left(\begin{matrix} 
16u+6y+2z+8 & 4u+2y+4 \\ 4u+2y+4 & 8 \end{matrix}\right).
$$ 
Note that $\det(H_{u,v}(y,z))=4(12-2u-y)(2u+y)+16(z-y)+48>0$. Therefore, we see $H_{u,v}(y,z)$ 
is positive definite, and so $\tilde{g}_1$ is convex in $y$ and $z$. 
Therefore we only need to consider the vertices of the polygon defined by 
$$
v\le y\le \min\{2v-8/3,\,3\},\quad y\le z\le 3.
$$
That is, $(y,z)$ is one of
$$
(v,v),\quad (v,3),\quad \bigl(\min\{2v-8/3,3\},\min\{2v-8/3,3\}\bigr),\quad \bigl(\min\{2v-8/3,3\},3\bigr).
$$
Note that $y\le \min\{2v-8/3,\,3\}$ means we have $y\le 2v-8/3$ for $8/3\le v\le 17/6$ and $y\le 3$ for
$v\ge 17/6$. We have 
\begin{enumerate}
\item[(a)] $\tilde{g}_1(v,v,v,v)=0$. 
\item[(b)] If $8/3\le v\le 3$ then 
$$
\tilde{g}_1(v,v,v,3)=-v(3-v)(3v-8)\le 0.
$$
\item[(c)] If $8/3\le v\le 17/6$ then  
$$
\tilde{g}_1(v,v,2v-8/3,2v-8/3)=\frac{-4(3v-8)(9v^3-72v^2+105v+76)}{27}\le0.
$$
\item[(d)] If $8/3\le v \le 17/6$ then 
$$
\tilde{g}_1(v,v,2v-8/3,3)=\frac{-4(3v-8)(9v^3-54v^2+78v+8)}{27}\le 0.
$$
\item[(e)] If $17/6\le v\le 3$ then
$$
\tilde{g}_1(v,v,3,3)=-2(3-v)(2v^3+2v^2-15v-19)\le 0.
$$
\end{enumerate}
The statements (c) and (e) give
$$
\tilde{g}_1\bigl(v,v,\min\{2v-8/3,3\},\min\{2v-8/3\}\bigr)\le 0.
$$
The statements (d) and (e) give
$$
\tilde{g}_1\bigl(v,v,\min\{2-8/3,3\},3\bigr)\le 0.
$$
By convexity, this means that $\tilde{g}_1(v,v,y,z)\le 0$. Hence,
$$
\tilde{F}(u,v,y,z)-\tilde{F}(u,v,y,y)
=(z-y)\tilde{g}_1(u,v,y,z)\le (z-y)\tilde{g}_1(v,v,y,z)\le 0.
$$
Therefore, $\tilde{F}(u,v,y,z)\le\tilde{F}(u,v,y,y)$ as claimed in part (a) of the lemma.

\medskip

We now prove part (b) of the lemma which says $\tilde{F}(u,v,y,y)\le\tilde{F}(u,v,x,x)$. 
We calculate
\begin{eqnarray*}
\frac{\partial\tilde{F}(u,v,y,y)}{\partial y}
& = & 4y^3+24(u+1)y^2 \\
&& \quad +4\bigl( 8u^{2}-2\left(v^{2}+6\right) u -17 v^{2}-8\bigr)y \\
&& \quad -16 v^{2} u^{2}+8\left(v^{2}+4\right) u +24 v^{4}+24 v^{2}.
\end{eqnarray*}
We claim that this function is non-positive and so $\tilde{F}(u,v,y,y)$ is a decreasing function of $y$.

Since
$$
\frac{\partial^3\tilde{F}(u,v,y,y)}{\partial y^3}=24(2u+y+2)>0,
$$
we see $\partial\tilde{F}(u,v,y,y)/\partial y$ is convex in $y$ and so when $u \le y\le 3$ we have
$$
\frac{\partial\tilde{F}(u,v,y,y)}{\partial y}
\le \max \left\{\frac{\partial\tilde{F}(u,v,y,y)}{\partial y}\Big|_{y=u}, \, 
\frac{\partial\tilde{F}(u,v,y,y)}{\partial y}\Big|_{y=\min\{2v-8/3, 3\}}\right\}.
$$
Note that when $v\le 17/6$ we have $2v-8/3\le 3$.

Initially, using $v\le u\le 3/2+v^2/6$, we have
\begin{eqnarray*}
\frac{\partial\tilde{F}(u,v,y,y)}{\partial y}\Big|_{y=u}
& = & 12(u^2-v^2)(5u-2v^2-2) \\
& \le & 2(u^2-v^2)(33-7v^2)\le 0,
\end{eqnarray*}
when $\sqrt{33/7}\le v\le u\le 3$. 

Second, using $y\le 2v-8/3$, $8/3\le v\le 17/6$, and $u\ge v$, we obtain
\begin{eqnarray*}
  27\times\frac{\partial\tilde{F}(u,v,y,y)}{\partial y}\Big|_{y=2v-8/3}&=& 9u^2(-48 v^{2}+192 v -2)\\
  &&+9u(-48 v^{3}+376 v^{2}-1056 v +992)\\
  &&+8(81 v^{4}-351 v^{3}+585 v^{2}-504 v +608).
\end{eqnarray*}
We observe that $-48 v^{2}+192 v -2<0$ when $2\le v\le 17/6$, and we then regard this expression as a quadratic function of $u$. This function has the axis of symmetry
$$
u=\frac{-6 v^{3}+47 v^{2}-132 v +124}{12 v^{2}-48 v +64}<0,
$$
when $2\le v\le 17/6$. This implies that
\begin{eqnarray*}
\frac{\partial\tilde{F}(u,v,y,y)}{\partial y}\Big|_{y=2v-8/3}&\le&\frac{\partial\tilde{F}(u,v,y,y)}{\partial y}\Big|_{y=2v-8/3, u=v}  \\
&=&\frac{-8 \left(3 v -8\right) \left(9 v^{3}-72 v^{2}+105 v +76\right)}{27},
\end{eqnarray*}
which is non-positive when $8/3\le v\le 17/6$.

Finally, using $17/6\le v\le 3$ and $u\ge v$ we get
\begin{eqnarray*}
\frac{\partial\tilde{F}(u,v,y,y)}{\partial y}\Big|_{y=3}
& = & -16u^2(v^2-6)-8u(2v^2-13)+24v^4-180v^2+228 \\
& \le & -16v^2(v^2-6) -8v(2v^2-13)+24v^4-180v^2+228 \\
& = & 4 \left(v -3\right) \left(2 v^{3}+2 v^{2}-15 v -19\right)\le 0.
\end{eqnarray*}

Hence, for $y\le \min\bigl\{2v-8/3, 3\bigr\}$ and $8/3\le v\le 3$, we have $\partial\tilde{F}(u,v,y,y)/\partial y\le 0$. Therefore, the function $\tilde{F}(u,v,y,y)$ is decreasing in $y$, and thus 
$$
\tilde{F}(u,v,y,y)\le \tilde{F}(u,v,x,x),
$$
for $y$ and $v$ in this range.

\end{proof}

\medskip

\begin{proof}(Lemma~\ref{lemma-2}.)
We suppose $\sqrt{23/3}\le v\le 17/6$ and $2v-8/3\le y\le z\le 3$.

We begin by proving part (a) of the lemma.
We have
\begin{eqnarray*}
\frac{1}{2}\frac{\partial ^2 \tilde{F}}{\partial z^2}(u,v, y,z)
& = &  (2u+y)^2-12v^2+12z-12 \\
& \ge & (2v+y)^2-12v^2+12y-12 \\
& \ge & (4v-8/3)^2-12v^2+24v-32-12 \\
& = & 16v^2-64v/3+64/9-12v^2+24v-44 \\
& = & (4/9)\bigl(9v^2 +6v-83\bigr)\ge0. 
\end{eqnarray*}
Therefore $\tilde{F}(u,v,y,z)$ is convex in $z$, and so
$$
\tilde{F}(u,v,y,z)\le \max\bigl\{\tilde{F}(u,v,y,y),\tilde{F}(u,v,y,3)\bigr\}.
$$
This proves part (a).

We now prove part (b). We obtain
\begin{eqnarray*}
\frac{1}{4} \frac{\partial^2\tilde{F}(u,v,y,y)}{\partial y^2} &=&  3 y^{2}+12\left(u+1\right) y +8 u^{2}-2\left( v^{2}+6\right) u -17 v^{2}-8\\
 &\ge &3 (2v-8/3)^{2}+12\left(u+1\right)(2v-8/3)\\
 &&\qquad\qquad\qquad\qquad\quad+8 u^{2}-2\left( v^{2}+6\right) u -17 v^{2}-8\\
 &=&(1/3)\bigl(24 u^{2}-6\left(v^{2}-12 v +22\right) u -15 v^{2}-24 v -56\bigr)\\
 &\ge&(1/3)\bigl(24 v^{2}-6\left(v^{2}-12 v +22\right) v -15 v^{2}-24 v -56\bigr)\\
 &=&(1/3)\bigl(-6 v^{3}+81 v^{2}-156 v -56\bigr).
\end{eqnarray*}
This is positive when $\sqrt{23/3}\le v\le 3$. Hence, for $\sqrt{23/3}\le v\le 17/6$ the function $\tilde{F}(u,v,y,y)$ is convex in $y$, which yields part (b).

We now prove part (c). 
Using $y\ge x\ge v$, we compute
\begin{eqnarray*}
        \frac{1}{2}\,\frac{\partial^2 \widetilde{F}(u,v,y,3)}{\partial y^2}
        & =  & (2u+3)^2-12v^2+12y-12 \\
        & \ge & (2v+3)^2-12v^2+12v-12 \\
        & = & \frac{2(17-6v)(6v-1)+7}{9}>0,
\end{eqnarray*}
when $1/6\le v\le 17/6$.
Therefore $\widetilde{F}(u,v,y,3)$ is a convex function of $y$. This gives part (c).
\end{proof}

\medskip

Before we prove Lemma~\ref{lemma-8} we show that $F(w,2u-w,2u-w,3)$ and $F(w,2u-w,3,3)$ are
convex functions of $u$ on the domain in question.

\begin{lem}\label{lemma-7}
    For $1\le w\le u\le  3$, we have
    $$
    \frac{\partial^2F(w, 2u-w, 3, 3)}{\partial u^2}\ge 0.
    $$
    Therefore $F(w,2u-w,3,3)$ is a convex function of $u$.
\end{lem}

\begin{proof}
    Direct computation yields
\begin{eqnarray*}
    \frac{1}{8}\frac{\partial^2 F(w, 2u-w, 3, 3)}{\partial u^2}&=& 24w^{2} u^{2} +\left(-36 w^{3}-12 w^{2}-72 w +24\right) u \\
    &&\qquad\qquad\qquad+13 w^{4}+4 w^{3}+85 w^{2}-26 w +24.
\end{eqnarray*}
When this is regarded as a quadratic function of $u$, it attains its minimum at $u=u_{\mathrm{min}}$, where
$$
u_{\mathrm{min}}=\frac{3 w^{3}+w^{2}+6 w -2}{4 w^{2}}. 
$$
and we have $u_{\rm min}\le w\le u$ when $26/9\le w\le 3$. 

We compute
$$
  \frac{1}{4}\frac{\partial^2 F(w, 2u-w, 3, 3)}{\partial u^2}\big|_{u=u_{\mathrm{min}}}
  = -w^{4}-10 w^{3}+59 w^{2}-52 w -48+\frac{72}{w}-\frac{12}{w^{2}},
$$
which is positive when $1\le w\le 26/9$, and 
\begin{eqnarray*}
 \frac{1}{8}\frac{\partial^2 F(w, 2u-w, 3, 3)}{\partial u^2}\big|_{u=w}&=& \left(3 -w\right) \left(-w^{3}+5 w^{2}+2 w +8\right)\ge 0,  
\end{eqnarray*}
when $23/9\le w\le 3$ with equality if and only if $w=3$. 
This completes the proof.
\end{proof}

\medskip

\begin{lem}\label{lem-7.5}
    For fixed $23/9\le w\le 3$, we have
    \[
    \frac{\partial^2 F(w,2u-w,2u-w,3)}{\partial u^2}\ge 0
    \]
    whenever $\max\{w,25/9\}\le u\le 3$.
    Therefore, $F(w,2u-w,2u-w,3)$ is a convex function of $u$ on this interval.
\end{lem}

\begin{proof}
A direct calculation shows that
\begin{eqnarray*}
F(w,2u-w,2u-w,3) & = & 16(w-1)^2u^4 -16(3w^3-7w^2+12w-10)u^3 \\
&& \quad +4(13w^4-34w^3+88w^2-62w-16 )u^2 \\
&& \quad -4(6w^5-17w^4+57w^3-34w^2+12 w-24)u \\
&& \quad +4w^6-12w^5+49w^4-22w^3+37w^2-24w.
\end{eqnarray*}
We have
\begin{eqnarray*}
\frac{\partial^2 F(w,2u-w,2u-w,3)}{\partial u^2} 
& = & 192(w-1)^2u^2-96(3w^3-7w^2+12w-10)u \\
&& \quad +8(13w^4-34w^3+88w^2-62w-16).
\end{eqnarray*}
This is a quadratic function of $u$ that has its minimum at
$$
u=u_{\rm min}=\frac{3w^3-7w^2+12w-10}{4(w-1)^2},
$$
and we have $u_{\mathrm{min}}\le 25/9\le u$ when $23/9\le w\le 25/9$ and $u_{\mathrm{min}}\le w\le u$ when $25/9\le w\le 3$. We then compute 
\begin{eqnarray*}
    \lefteqn{\frac{\partial^2 F(w,2u-w,2u-w,3)}{\partial u^2} \big|_{u=25/9}}\\
    &=&\frac{2808 w^4 - 28944 w^3 + 109408 w^2 - 179792 w + 108544}{27}>0,
\end{eqnarray*}
when $23/9\le w\le 25/9$, and
$$
    \frac{\partial^2 F(w,2u-w,2u-w,3)}{\partial u^2} \big|_{u=w}
    =8 w^{4}+16 w^{3}-256 w^{2}+464 w -128>0,
$$
when $25/9\le w\le 3$.

Hence $F(w,2u-w,2u-w,3)$ is a convex function of $u$ when $23/9\le w\le 3$. 
\end{proof}

\begin{proof}(Lemma~\ref{lemma-8})
We suppose $23/9\le w\le 3$. 

We begin with part (a) of the lemma. A direct computation yields
\begin{eqnarray*}
    \lefteqn{F(w,2u-w,2u-w,2u-w)} \\
    & = & 4(3-w)(u-w)^2\bigl((23-9w)-(3-w)(9-(2u-w)^2)\bigr) \\
    & \le & 4(3-w)(u-w)^2(23-9w).
\end{eqnarray*}
Therefore $F(w,2u-w,2u-w,2u-w)\le 0$ with equality if and only if $u=w$ or 
$(w,u)=(23/9,25/9)$.

We now consider part (b) and we write $x=2u-w$ where 
\begin{equation}\label{eq-boundary-b}
\max\bigl\{w,\,25/9\bigr\}\le u\le (w+3)/2.
\end{equation}
From Lemma~\ref{lem-7.5} we see that $F(w,2u-w,2u-w,3)$ is a convex function of $u$.
We then compute
\begin{eqnarray*}
\lefteqn{ F(w,2u-w,2u-w,3)\Big|_{u=25/9} }\\
& = & 
\frac{(9w - 23)(2916w^5 - 49896w^4 + 338409w^3 - 1127915w^2 + 1844368w - 1185200)}{6561},
\end{eqnarray*}
which is non-positive when $23/9\le w\le 25/9$ and zero if and only if $w=23/9$,
$$
F(w,2u-w,2u-w,3)\Big|_{u=w}=-w(3-w)^2(3w-8),
$$
which is non-positive when $25/9\le w\le 3$ with equality if and only if $w=3$, and
$$
F(w,2u-w,2u-w,3)\Big|_{u=(w+3)/2}=-(3-w)^3(9w-23),
$$
which is non-positive when $23/9\le w\le 3$ with equality if and only if $w=23/9$ or $w=3$.
Therefore $F(w,2u-w,2u-w,3)$ is non-positive on the boundary of the region 
\eqref{eq-boundary-b}, which completes the proof of part (b) of the lemma.

We now consider part (c) of the lemma. We again write $x=2u-w$ and we suppose
\begin{equation}\label{eq-boundary-c}
\max\bigl\{w,\,(5+3\sqrt{33})/8\bigr\}\le u\le (w+3)/2.
\end{equation}
According to Lemma~\ref{lemma-7}, for given $23/9 \le w\le 3$, the function 
$F(w, 2u-w, 3, 3)$ is a convex function in $u$.
We then compute that
\begin{eqnarray*}
    F(w, 2u-w, 3, 3)\Big|_{u=(w+3)/2}&=&-\left(9 w -23\right) \left(3 -w\right)^{3}\le 0,
\end{eqnarray*}
with equality if and only if $w=23/9$ or $w=3$, 
\begin{eqnarray*}
    F(w, 2u-w, 3, 3)\big|_{u=w}&=&-\left(w +1\right) \left(3 - w\right)^{2}\left(4 w^{2}-5 w -17\right) \le0,
\end{eqnarray*}
which is non-positive when $(5+3\sqrt{33})/8\le w\le 3$ with equality if and only if $w=3$ or $w=(5+3\sqrt{33})/8$, and
\begin{eqnarray*}
    F(w, 2u-w, 3, 3)\big|_{u=\bigl(5+3\sqrt{33} \bigr)/8}
    &=& \frac{( 5+3 \sqrt{33}-8 w)^{2}\, h(w)}{128},
\end{eqnarray*}
where
\begin{eqnarray*}
    h(w)&=&8 w^{4}-\bigl(12 \sqrt{33}+20\bigr) w^{3}+\bigl(12 \sqrt{33}+302\bigr) w^{2}\\
    &&\qquad\qquad\qquad-\bigl(102 \sqrt{33}+102\bigr) w +33 \sqrt{33}+289,
\end{eqnarray*}
which is negative when $23/9\le w\le (5+3\sqrt{33})/8$. Therefore $F(w,2u-w,3,3)$ is non-positive
on the boundary of the region \eqref{eq-boundary-c}. This completes the proof of part (c).
\end{proof}

\begin{proof}(Lemma~\ref{lemma-11})
Observe that $\bigl(1+2\cos(2\pi/p)\bigr)$ for $p\in{\mathbb Z}\cup \{\infty\}$ with $3\le p$ is
a monotone increasing sequence lying in $[0,3]$. We examine the possible values to five decimal places, since in one instance we need to examine the fourth decimal place.
First, note that
$$
50/9 = 5.55555\ldots, \qquad 
(5+3\sqrt{33})/4= 5.55842\ldots. 
$$
In what follows, write 
$$
U(p_1,p_2)=\bigl(1+2\cos(2\pi/p_1)\bigr)+\bigl(1+2\cos(2\pi/p_2)\bigr) .
$$
We need to show that for all pairs $p_1\le p_2$ either $U(p_1,p_2)<50/9$ or 
$U(p_1,p_2)>(5+3\sqrt{33})/4$.

If $p_1\le 9$ then
$$
U(p_1,p_2)\le U(9,\infty)= 5.53208\ldots  <  50/9.
$$
If $14\le p_1\le p_2$ then
$$
U(p_1,p_2)\ge U(14,14)= 5.60387\ldots > (5+3\sqrt{33})/4.
$$
For $p_1=10,\,\ldots,\,13$ we have
\begin{eqnarray*}
U(10,25) = 5.55520\ldots <50/9 & < & (5+3\sqrt{33})/4< 5.55991\ldots =U(10,26), \\
U(11,17) = 5.54745\ldots  <50/9 & < & (5+3\sqrt{33})/4< 5.56189\ldots =U(11,18), \\
U(12,14) = 5.53398\ldots <50/9 & < & (5+3\sqrt{33})/4< 5.55914\ldots =U(12,15), \\
U(13, 13) = 5.54182\ldots <50/9 & < & (5+3\sqrt{33})/4<  5.57284\ldots = U(13,14).
\end{eqnarray*}
\end{proof}
\medskip

\section{The main results for small values of $p_1$ and $p_2$.}\label{sec-small-p1-p2}

We now turn our attention to the case when $w+w\le 50/9$. 
Note that our hypothesis implies we have
$$
w\le (w+x)/2\le  25/9.
$$
Since $1+2\cos(2\pi/13)<25/9<1+2\cos(2\pi/14)$, we are only concerned with $p_1\le 13$. 
The main result is:

\begin{prop}\label{prop-main-small}
Let $p_1$, $p_2$, $p_3$, $p_4$ be integers satisfying 
$3\le p_1\le p_2\le p_3\le p_4\le \infty$. Define
$$
    \begin{array}{ll}
    w=1+2\cos(2\pi/p_1), \quad & 
    x=1+2\cos(2\pi/p_2), \quad \\
    y=1+2\cos(2\pi/p_3), \quad & 
    z=1+2\cos(2\pi/p_4).
    \end{array}
$$
Let $\Gamma(p_1,p_2,p_3;p_4)\subset{\rm SU}(2,1)$ be the corresponding triangle group.
For given $p_1$ and $p_2$ for which $w+x\le 50/9$ there exist $p'_3$ and $p'_4$ so that 
for all $p_3\ge p'_3$ and $p_4\ge p'_4$, in $\Gamma(p_1,p_2,p_3;p_4)$ the
element $R_1R_2R_3$ is loxodromic. Moreover
\begin{enumerate}
    \item[(a)] We may take $p'_3=p'_4$ so that 
    $$
    \max\left\{2,\,x,\,\frac{9(w+x)-8}{14}\right\}\le y\le z. 
    $$
    \item[(b)] When $p_4=\infty$, so $z=3$, we may take $p'_3$ so that 
    $$
    \max\left\{1,\,x,\,\frac{3(w+x)}{2}-\frac{16}{3}\right\}\le y.
    $$
\end{enumerate}
\end{prop}

In order to prove Proposition~\ref{prop-main-small} we use two techniques, both of which follow from the definition of $F=f_1^2+4f_3^3$. First, it is clear that 
if $f_2>0$ then $F>0$. Secondly, we show that if there are parameters for which both $f_1>0$ and $F>0$ then $F$ is an increasing function of $y$ and $z$. 

\begin{lem}\label{lem-in-SU21}
    Using the notation of Proposition~\ref{prop-main-small}, if $\Gamma(p_1,p_2,p_3;p_4)$
    is a subgroup of ${\rm SU}(2,1)$ then $y+z>2$. 
\end{lem}

\begin{proof}
Using Lemma~\ref{lem-H-sig} that if $\Gamma$ lies in ${\rm SU}(2,1)$ then $f_2(w,x,w,z)>-1$.
That is,
\begin{eqnarray*}
    0 & < & f_2(w,x,y,z)+1 \\
    & = & (1-w)(x-1)+y+z-2.
\end{eqnarray*}
    If $x>1$ then $z\ge y\ge x>1$ and $y+z>2$. If $x\le 1$ then $w\le x\le 1$ and 
    $(1-w)(x-1)\le 0$. Hence $y+z>2$. 
\end{proof}

\medskip

We begin with the case where $p_1\le 6$, which is equivalent to $w\le 2$.

\begin{lem}\label{lem w<2}
    Suppose that $w\le 2$. If $y+z\ge 4$ then $F(w,x,y,z)\ge 0$. In particular, 
    if $2\le y\le z$ then $F(w,x,y,z)\ge 0$ and if $y\ge 1$ then $F(w,x,y,3)\ge 0$.
\end{lem}

\begin{proof}
We consider
$$
f_2(w,x,y,z)=-wx+w+x+y+z-4.
$$
Since $F=f_1^2+4f_2^3$, if $f_2\ge 0$ then $F\ge 0$. Hence it suffices to prove $f_2(w,x,y,z)\ge 0$.

First suppose that $w\le 2\le x\le y\le z$.
$$
f_2(w,x,y,z)=(2-w)(x-2)+(y-w)+(z-x)\ge 0
$$
with equality if and only if $w=x=y=z=2$.

Now suppose that $w\le x\le 2$. Then
$$
f_2(w,x,y,z)=w(2-x)+(x-w)+y+z-4
\ge y+z-4.
$$
Thus, if $y+z\ge 4$ then $f_2(w,x,y,z)\ge 0$.
\end{proof}

\medskip

This has proved Proposition~\ref{prop-main-small} in the case
where $p_1\le 6$. We now consider $p_1\ge 6$, which implies $w\ge 2$. In this case we will use the following lemma to determine when $R_1R_2R_3$ is loxodromic or parabolic in $\Gamma(p_1,p_2,p_3;p_4)$, that is when $F(w,x,y,z)> 0$.

\begin{lem}\label{loxo condition}
For given $w\le x$ suppose there exist $y'$, $z'$ with $x\le y'\le z'$ and also $w+x+y'\ge 6$ 
(in particular this holds if $w\ge 2$) such that
$$
f_1(w, x, y', z')\ge 0 \qquad \text{and} \qquad
F(w, x, y', z')\ge 0.
$$
Then $F(w, x, y, z)\ge 0$ for all $y\ge y'$ and $z\ge z'$. 
\end{lem}

\begin{proof}
By definition, $f_1(w,x,y,z)$ and $f_2(w,x,y,z)$ are linear functions of $y$ and $z$. Moreover,
\begin{eqnarray*}
\frac{\partial f_1}{\partial y}
=z+w+x-6,& \quad &
\frac{\partial f_2}{\partial y}
=1, \\
\frac{\partial f_1}{\partial z}
=y+w+x-6,& \quad &
\frac{\partial f_2}{\partial z}
=1.
\end{eqnarray*}
When $6\le w+x+y\le w+x+z$ 
these partial derivatives are all non-negative. Using Theorem~\ref{Goldman} we see
$$
F=f_1^2+4f_2^3,
$$
and so 
\begin{eqnarray*}
\frac{\partial F}{\partial y}
& = & 2f_1\,\frac{\partial f_1}{\partial y} 
+12f_2^2\,\frac{\partial f_2}{\partial y}, \\
\frac{\partial F}{\partial z}
& = & 2f_1\,\frac{\partial f_1}{\partial z} 
+12f_2^2\,\frac{\partial f_2}{\partial z}.
\end{eqnarray*}
Hence, if $f_1(w,x,y,z)\ge 0$ then $F(w,x,y,z)$ is an increasing function of both $y$ and $z$.
\end{proof}

\medskip

We can now prove Proposition~\ref{prop-main-small}~(a) in the case $w\ge 2$.

\begin{prop}\label{prop-y=z}
Suppose we are given $w$ and $x$ with 
$$
2\le w\le x\le 3,\quad w+x\le 50/9.
$$
Then for all 
$$
\max\left\{x,\,\frac{9(w+x)-8}{14}\right\}\le y\le z.
$$
we have $F(w,x,y,z)\ge 0$ with equality if and only if $(w,x,y,z)=(2,2,2,2)$ or 
$(23/9,3,3,3)$.
\end{prop}

\begin{proof}
We will use Lemma~\ref{loxo condition} with 
$$
y'=z'=\max\left\{x,\,\frac{9(w+x)-8}{14}\right\}.
$$
We observe that $\bigl(9(w+x)-8\bigr)/14= x$ if and only if 
$$
9w-5x-8 = 0.
$$

First consider the case where $\bigl(9(w+x)-8\bigr)/14\le x$, that is when 
$$
9w-5x-8\le 0.
$$
Since $x\le 3$ we have $w\le 23/9$. In this case, we consider
$$
F(w,x,x,x)=(3-w)(x-w)^2\bigl((3-w)x^2-4\bigr).
$$
When $2\le w\le 23/9$ and $x\ge (9w-8)/5$ we have
\begin{eqnarray*}
    (3-w)x^2 -4 & \ge & \frac{(3-w)(9w-8)^2}{25}-4 \\
    & = & \frac{(w-2)(23-9w)(9w-2)}{25}\ge 0
\end{eqnarray*}
with equality if and only if $w=2$ or $w=23/9$.

Now we consider the case where $\bigl(9(w+x)-8\bigr)/14\ge x$, that is 
$$
9w-5x-8\ge 0.
$$
Therefore, we want to show that when 
$$
w\le x\le \min\bigl \{(9w-8)/5, \,50/9-w\bigr\}
$$
and $y=z=\bigl(9(w+x)-8\bigr)/14$ then $F(w,x,y,z)\ge 0$. 

To simplify the calculations, we write $w=2+7a-2\sqrt{b}$ and $x=2+7a+2\sqrt{b}$. On the line $w=x$ we have $b=0$ and on the line $9w-5x-8=0$
we have $b=a^2$. On the line $w+x=50/9$ we have $a=1/9$. Therefore when 
$$
w\le x\le \min\bigl \{(9w-8)/5, \,50/9-w\bigr\}
$$
we have $0\le a\le 1/9$ and $0\le b\le a^2$. 
Also, 
$$
\frac{9(w+x)-8}{14}=2+9a.
$$

First consider
$$
f_1(2+7a-2\sqrt{b},2+7a+2\sqrt{b},2+9a,2+9a) 
= a^2(76-630a)+(a^2-b)(12-56a).
$$
This is clearly non-negative when $0\le a\le 1/9$ and $0\le b\le a^2$ and is zero
if and only if $a=b=0$. 

Now, define
\begin{eqnarray*}
h(a,b) & = & F(2+7a-2\sqrt{b},2+7a+2\sqrt{b},2+9a,2+9a) \\
& = & 256b^3 - 6272a^2b^2 - 576ab^2 + 144b^2 \\
&& \quad + 38416a^4b + 7504a^3b - 1344a^2b - 5488a^5 - 1664a^4 + 256a^3.
\end{eqnarray*}
We calculate
$$
\frac{\partial^2 h(a,b)}{\partial b^2} = 32(48b-392a^2-36a+9).
$$
Since the quadratic polynomial $-392a^2-36a+9$ is positive for $0\le a\le 1/9$, we see that $\partial^2 h(a,b)/\partial b^2>0$ and so $\partial(h(a,b)/\partial b$ is an increasing function. We calculate
$$
\frac{\partial h(a,b)}{\partial b}\Big\vert_{b=a^2} = 16a^2(1665a^2+397a-66).
$$
This is negative for all $0\le a\le 1/9$ and so $h(a,b)$ is a decreasing function of $b$.
Therefore for $0\le a\le 1/9$ and $0\le b\le a^2$ we have 
$$
h(a,b)\ge h(a,a^2) =16a^3(1-9a)(1-5a)(45a + 16)\ge 0,
$$
with equality if and only if $a=0$ or $a=1/9$.

Thus we have shown that when $2\le w\le x$, $w+x\le 50/9$ and $9w-5x-8\ge 0$ we have
\begin{eqnarray*}
    f_1\left(w,x,\frac{9(w+x)-8}{14}, \frac{9(w+x)-8}{14}\right) & \ge & 0, \\
    F\left(w,x,\frac{9(w+x)-8}{14}, \frac{9(w+x)-8}{14}\right) & \ge & 0.
\end{eqnarray*}
The result follows from Lemma~\ref{loxo condition}.
\end{proof}

\medskip

This has proved Proposition~\ref{prop-main-small}~(a) in the case where $w\ge 2$. 
We now turn our attention to finding proving Proposition~\ref{prop-main-small}~(b) in the case
where $w\ge 2$. That is, we find a bound on $y'$ when $z=3$. Again we use Lemma~\ref{loxo condition}. We will show that we can take
$$
y'=\max\left\{x,\,\frac{3(w+x)}{2}-\frac{16}{3}\right\}.
$$
Observe that these two quantities are equal when $3w+x=32/3$. 

\begin{prop}\label{prop-z=3}
Suppose we are given $w$ and $x$ with
$$
2\le w\le x\le 3,\quad w+x\le 50/9.
$$
Then for all
$$
\max\left\{x,\,\frac{3(w+x)}{2}-\frac{16}{3}\right\}\le y
$$
we have $F(w,x,y,3)\ge 0$ with equality if and only if $(w,x,y)=(8/3,8/3,8/3)$ or $(23/9,3,3)$.    
\end{prop}

We will prove this using Lemma~\ref{loxo condition} and we break the proof down into three lemmas.

\begin{lem}\label{lem-f1-loxo-cond}
Suppose that we are given $w$ and $x$ so that
$$
2\le w\le x\le 3,\quad w+x\le 50/9.
$$
If 
$$
y'=\max\left\{x,\,\frac{3(w+x)}{2}-\frac{16}{3}\right\}
$$ 
then $f_1(w,x,y',3)\ge 0$ with equality if and only if $w=x=y'=2$. 
\end{lem}

\begin{proof}
Setting $u=(w+x)/2$, we have $x=2u-w$.
Note that
\[
\max\bigl\{x,\,3(w+x)/2-16/3\bigr\}
=\max\bigl\{2u-w,\,3u-16/3\bigr\}
=2u-w,
\]
when $u\le 16/3-w$, and
\[
\max\bigl\{2u-w,\,3u-16/3\bigr\}
=3u-16/3,
\]
when $u\ge 16/3-w$.

For each case, we have
\begin{eqnarray*}
 f_1(w,2u-w,2u-w,3)
 &=&
 4(1-w)u^2+2(w^2+6w-6)u-7w^2+3w-2,\\
 f_1(w,2u-w,3u-16/3,3)
 &=&
 (6-4w)u^2+(2w^2+14w-77/3)u-7w^2+14.
\end{eqnarray*}
When $w\ge 2$, both expressions are clearly concave functions of $u$. Thus we only need to show $f_1$ is positive on the boundaries of the regions in the statement.

Therefore, the boundary of the region is made up of the following arcs, 
and we show on each one that $f_1\ge 0$.
\begin{enumerate}
\item[(a)] $w=2$, $2\le x\le 3$, $y'=x$:
$$
f_1(2,x,x,3) = (x-2)(4-x)\ge 0.
$$
\item[(b)] $2\le w\le 23/9$, $x=3$, $y'=3$:
$$
f_1(w,3,3,3)=-3w^2+12w-11>0.
$$
\item[(c)] $2\le w\le 8/3$, $x=w$, $y'=w$:
$$
f_1(w,w,w,3)=-(w-2)(2w^2-5w-1)\ge 0.
$$
\item[(d)] $8/3\le w\le 25/9$, $x=w$, $y'=3w-16/3$: 
$$
f_1(w,w,3u-16/3,3)
=\frac{(3-w)(6w^2-21w+14)}{3}>0.
$$
\item[(e)] $23/9\le w\le 8/3$, $x=32/3-3w$, $y'=32/3-3w$:
$$
f(w,32/3-3w,32/3-3w,3)=\frac{-(2w-5)(27w^2-105w+86)}{9}>0.
$$
\item[(f)] $23/9\le w\le 25/9$, $x=50/9-w$, $y'=3$:
$$
f_1(w,50/9-w,3,3)=\frac{-117w^2+650w-891}{81}>0.
$$
\end{enumerate}
The only place where we have equality is $w=x=y'=2$. 
\end{proof}

\begin{lem}\label{lem-13} Suppose we are given $w$, $x$ with $2\le w\le x$ and $3w+x\le 32/3$. Then 
$F(w,x,x,3)\ge 0$
with equality if and only if $(w,x)=(8/3,8/3)$ or $(23/9,3)$.
\end{lem}

\begin{proof}
Note that $w\le x$ and $3w+x\le 32/3$ imply that $w\le 8/3$. Write $u=(w+x)/2$. 
Then
$$
u\le \min\left\{\frac{w+3}{2},\frac{16}{3}-w\right\},
$$
and these two quantities are equal when $w=23/9$.

When $2\le w\le 23/9$ and $w\le u\le (w+3)/2$, write
$$
w=3-a,\qquad
u=3-a+\frac{ab}{2}.
$$
Clearly, $4/9\le a\le 1$ and $0\le b\le 1$.
Thus,
\begin{eqnarray*}
h_1(a,b)&=&\frac{1}{531441a^2}F(w,2u-w,2u-w,3)\\
&=&(3a-1)(3-a)(1-b)^4+2(a^3-9a^2+23a-7)b(1-b)^3\\
&&\quad +(a^4-8a^3+4a^2+50a-19)b^2(1-b)^2\\
&&\quad +2(-3a^3+12a^2+5a-4)b^3(1-b)+a(9a-4)b^4.
\end{eqnarray*}
It is clear that $h_1(a,b)\ge0$ for all $a$ and $b$ in their respective ranges, with equality if and only if $(a,b)=(4/9,1)$, which corresponds to $(w,u)=(23/9,25/9)$. 
\medskip

When $23/9\le w\le 8/3$ and $w\le u\le 16/3-w$, write
$$
w=\frac{24-a}{9},
\qquad
u=\frac{24-a+2ab}{9}.
$$
Clearly, $0\le a\le 1$ and $0\le b\le 1$.
We have
\begin{eqnarray*}
\lefteqn{
 h_2(a,b)
}   \\
    &=&\frac{1}{531441a}F(w,2u-w,2u-w,3)\\
    &=& 243(24-a)(3+a)^2(1-b)^4\\
    && \quad + 36(3+a)\bigl(1512+a(783-63a)+2a^3\bigr)b(1-b)^3 \\
    &&\quad + 2\bigl(87480+a(89343-15174a)+a^3(14931-804a)+8a^5\bigr)b^2(1-b)^2 \\
    &&\quad + 12\bigl(5832-1511a^3+a(5967-4914a)+386a^4+8a^3(1-a^2)\bigr)b^3(1-b)\\
    &&\quad + 9(1-a)\bigl(648-192a+741a(1-a)+a^3(136-16a)\bigr)b^4.
\end{eqnarray*}
Hence, $h_2(a,b)\ge 0$ 
with equality if and only if $a=0$ or $(a,b)=(1,1)$, which correspond to $(w,u)=(8/3,8/3)$ and $(w,u)=(23/9,25/9)$, respectively.
\end{proof}
\medskip

\begin{lem}\label{lem-14}
Suppose we are given $w$ and $x$ with 
$w\le x$, 
$3w+x\ge 32/3$ and $w+x\le 50/9$. Then 
$$
F\bigl(w,x,3(w+x)/2-16/3,3\bigr)\ge 0.
$$
We have equality if and only if $(w,x)=(8/3,8/3)$ or $(23/9,3)$.
\end{lem}

\begin{proof}
Write 
$$
a=\frac{9(w+x)}{2}-24,\quad b=\frac{9(x-w)}{4}.
$$
Therefore
$$
w=\frac{a-2b}{9}+\frac{8}{3},\quad x=\frac{a+2b}{9}+\frac{8}{3},\quad y=\frac{3(w+x)}{2}-\frac{16}{3}=
\frac{a}{3}+\frac{8}{3}.
$$
The conditions 
$$
w\le x,\quad w+x\le \frac{50}{9},\quad 3w+x\ge \frac{32}{3}
$$
translate to 
$$
0\le b\le a\le 1.
$$
Substituting into $F(w,x,y,3)$ and factoring out a common denominator gives
\begin{eqnarray*}
h_3(a,b)&=&531441F(w,x,y,z)  \\
& = & 256b^6 + (- 128a^2  - 1536a + 1872)b^4 \\
&& \quad + (16a^4 + 96a^3 + 5832a^2 - 6264a - 2592)b^2 \\
&& \quad  + 72a^5 + 189a^4 - 5346a^3 + 1701a^2 + 5832a.
\end{eqnarray*}
Differentiating with respect to $b$ and using $0\le b\le a\le 1$, we obtain 
\begin{eqnarray*}
\frac{\partial h_3(a,b)}{\partial b}
& = & 1536b^5 
- 512a^2b^3 - 6144ab^3 + 7488b^3 \\
&& \quad +32a^4b +  192a^3b + 11664a^2b - 12528ab - 5184b \\
& = & 1536(b^2-a^2)b^3 + (1024a^2 - 6144a + 7488)(b^2-a^2)b \\
&& \quad + (1056a^4-5952a^3+19152a^2-12528a-5184 )b \\
& \le & 0. 
\end{eqnarray*}
Therefore $h_3$ is a decreasing function of $b$ and so
\begin{eqnarray*}
h_3(a,b)&\ge&h_3(a,a)\\
& = & 9a(1-a)(648+549a-741a^2+136a^3-16a^4) \\
& \ge &0.
\end{eqnarray*} 
\end{proof}

\medskip

\begin{proof}(Proposition~\ref{prop-z=3}) 
    We use Lemma~\ref{loxo condition}. Using Lemma~\ref{lem-f1-loxo-cond} we see that when 
    $$
    2\le w\le x\le 3,\quad w+x\le 50/9,\quad y'=\max\bigl\{x,\,3(w+x)/2-16/3\bigr\}
    $$
    then $f_1(w,x,y',3)\ge 0$. Similarly, when $x\ge 3(w+x)/2-16/3$, Lemma~\ref{lem-13} implies that $F(w,x,x,3)\ge 0$ and when $x\le 3(w+x)/2-16/3$ then Lemma~\ref{lem-14} implies that $F\bigl(w,x,3(w+x)/2-16/6,3\bigr)\ge 0$.
    Thus, $F(w,x,y,3)\ge 0$ for all $y\ge y'$,
    with equality if and only if $(w,x,y)=(8/3,8/3,8/3)$ or $(23/9,3,3)$. 
\end{proof}

\medskip

\section{Values of $(p_1,p_2,p_3,p_4)$ where $R_1R_2R_3$ is loxodromic.}

The lower bounds on $p_3$ and $p_4$ (that is on $y$ and $z$) given in 
Proposition~\ref{prop-main-small}~(a) and (b) are the best linear functions of $(w+z)$,
but they are not optimal. In this section 
we give direct calculations of the values of $p_1$, $p_2$, $p_3$, $p_4$ for which $W_B=R_1R_2R_3$ is loxodromic in $\Gamma(p_1,p_2,p_3;p_4)$. 

We first consider the cases where $3\le p_1\le 7$. In  Lemmas~\ref{p=3} and~\ref{p=4, 5, 6, 7}, we again show $f_2(w,x,y,z)>0$ in order
to conclude that $F(w,x,y,z)>0$. 
 We also use Lemma~\ref{lem-H-sig}, which states that for the Hermitian form to have the correct signature, we require $f_2(w,x,y,z)>-1$. We will restrict our attention to 
 parameters $p_1$, $p_2$, $p_3$, $p_4$ where this inequality is satisfied.

\begin{lem}\label{p=3}
Suppose that $p_1=3$ and $f_2(w,x,y,z)>-1$. Then, apart from the groups $\Gamma(3,3,4;7)$ and $\Gamma(3,3,5;5)$, the element 
$R_1R_2R_3$ of the group $\Gamma(p_1,p_2,p_3;p_4)$ is loxodromic. 
\end{lem}

\begin{proof}
When $p_1=3$ we have $w=0$ and so
$$
f_2(0,x,y,z)=x+y+z-4.
$$
If $f_2(0,x,y,z)\le -1$ then the Hermitian form does not have the correct signature, and if $f_2(w,x,y,z)>0$ then $R_1R_2R_3$ is loxodromic.

If $4\le p_2$ and $5\le p_3\le p_4$ then
$$
f_2(0,x,y,z)=x+y+z-4\ge 1 + (1+\sqrt{5})-4>0.
$$
Therefore $f_2(0,x,y,z)>0$
Hence, we only need to consider the two cases $p_2=3$ and $p_2=p_3=4$.

If $4=p_2=p_3$ and $7\le p_4$ then
$$
f_2(0,x,y,z)=x+y+z-4\ge 2+\bigl(1+2\cos(2\pi/7)\bigr)-4>0.
$$
When $p_4=4$ then $f_2(0,x,y,z)=-1$ and the Hermitian form is degenerate. When $p=5$ or $6$
we can check by inspection that
$$
F\left(0,1,1,\frac{1+\sqrt{5}}{2}\right)=4\sqrt{5}-7>0,\quad
F(0,1,1,2)=9>0.
$$
Therefore, $R_1R_2R_3$ is loxodromic in $\Gamma(3,4,4;p_4)$ for $p_4\ge 5$.

This leaves the case where $p_2=3$. Here we have 
$$
f_2(0,0,y,z)=y+z-4.
$$
This is non-negative when $6\le p_3\le p_4$ and positive unless $6=p_3=p_4$. In the latter case, $F(0,0,2,2)=16$. Therefore $R_1R_2R_3$ is loxodromic in $\Gamma(3,3,p_3,p_4)$
for $6\le p_3\le p_4.$
Hence we need to consider the cases where $p_3=3,4,5$. 

If $p_3=3$ then 
$f_2(0,0,0,z)=z-4\le -1$
and so the Hermitian form has the wrong signature for all $z\le 3$.

If $p_3=4$ then $f_2(0,0,1,z)=z-3$.
This is always negative, and is less than or equal to $-1$ when $z\le 2$, that is $p_4\le 6$. In this case, 
$$
F(0,0,1,z)=4z^3-11z^2+8z-8.
$$
This is negative when $p_4=7$ and positive when $p_4\ge 8$.

If $p_4=5$ then
$$
f_2\left(0,0,\frac{1+\sqrt{5}}{2},z\right)=z-\frac{7-\sqrt{5}}{2}.
$$
This is positive when $p_4\ge 8$. In this case 
$$
F\left(0,0,\frac{1+\sqrt{5}}{2},z\right)
=4z^3-\frac{21-\sqrt{5}}{2}\,z^2/2 +4(1+\sqrt{5})z-2(5+\sqrt{5}).
$$
This is negative when $p_4=5$ and positive when $p_4\ge 6$.
\end{proof}

When $p_1=p_2=p_3=p_4$ we have $F(w,w,w,w)=0$ and so $R_1R_2R_3$ has a repeated eigenvalue.
It is not hard to show that $R_1R_2R_3$ is parabolic when $p_1=6$ and elliptic otherwise. 
Note that $f_2(w,w,w,w)=-(w-2)^2$, which implies $\Gamma$ is in ${\rm SU}(2,1)$ if and only
if $1<w<3$, that is $5\le p_1<\infty$. So to find groups where $R_1R_2R_3$ is loxodromic, it
suffices to suppose $p_4\ge p_1+1$. 

\medskip

\begin{lem}\label{p=4, 5, 6, 7}
Suppose that $p_1\in\{4,\,5,\,6,\,7\}$. Then, apart from the group $\Gamma(4,4,4,5)$, 
if 
$$
p_1\le p_2\le p_3\le p_4\quad \hbox{ and } \quad (p_1+1)\le p_4
$$
then in the group $\Gamma(p_1,p_2,p_3;p_4)$ the element $R_1R_2R_3$ is loxodromic. Moreover,
in these groups the Hermitian form has signature $(2,1)$. 
\end{lem}

\begin{proof}
Note that $p_1\ge 4$ means $w\ge 1$. We write
$$
f_2(w,x,w,z)=\frac{1}{2}\,(w-1)\bigl(y+z-2x)
+\frac{1}{2}\,(3-w)(y+z-2)-1\ge -1
$$
with equality if and only if $w=x=y=z=1$. 
Therefore, using Lemma~\ref{lem-H-sig}, if $p_1\ge 4$ then $H$ has signature (2,1) unless
$p_1=p_2=p_3=p_4=4$. 

Since we are assuming $1\le w\le x\le y\le z$ this expression for $f_2(w,x,y,z)$ is non-negative whenever 
\begin{equation}\label{eq-y+z-bd}
y+z\ge 2+\frac{2}{3-w}=\frac{2(4-w)}{3-w}.
\end{equation}
It is easy to check that, for $p_1\in\{5,6,7\}$, the inequality \eqref{eq-y+z-bd} holds
when $p_3\ge p_1$ and $p_4\ge (p_1+1)$. 

Now consider $p_1=4$ so $w=1$, then the expression \eqref{eq-y+z-bd} translates to 
$y+z\ge 3$. 
If $5\le p_3\le p_4$ then $y+z\ge 1+\sqrt{5}>3$.
If $p_3\ge 4$ and $p_4\ge 6$, which is $y\ge 1$ and $z\ge 2$ then we have $y+z\ge 3$
with equality if and only if 
$(w,x,y,z)=(1,1,1,2)$, so $(p_1,p_2,p_3,p_4)=(4,4,4,6)$.
In this case $f_1^2(1,1,1,2)=1$ and so $F(1,1,1,2)>0$. Therefore $R_1R_2R_3$ is loxodromic
when $p_1=4$ and $(p_1+1)\le p_4$ apart from when $(p_1,p_2,p_3,p_4)=(4,4,4,5)$. 
\end{proof}

\medskip

\medskip
We then consider the cases $p_1=8$ and $p_1=9$. The following result shows that there are finitely many $\Gamma(p_1, p_2, p_3; p_4)$ groups with $p_1=8$ or $p_1=9$ in which the element $R_1R_2R_3$ is elliptic or parabolic.

\begin{lem}\label{p1=8,9}
    The element $R_1R_2R_3$ is loxodromic in the groups $\Gamma(p_1, p_2, p_3; p_4)$ for $p_1 \in \{8, 9\}$, except for the following cases:
\begin{enumerate}
    \item[(i)] $\Gamma(8, 8, 8; 8)$, $\Gamma(8, 8, 8; 9)$, and $\Gamma(8, 9, 9; 9)$;
    \item[(ii)] $\Gamma(9, p_2, p_2+1; p_2+1)$ with $9 \le p_2 \le 16$;\medskip
    \item[(iii)] $\Gamma(9, p_2, p_2; p_4)$ with $p_2 \le p_4 \le p_2 + \delta(p_2)$, where
    \[
    \delta(p_2) = \begin{cases}
    3, & p_2 \in \{9, 10, 11, 12\}, \\
    2, & p_2 \in \{13, 14, 15, 16, 17\}, \\
    1, & p_2 \in \{18, 19, 20\}, \\
    0, & p_2 \in \{21, 22\}.
    \end{cases}
    \]
\end{enumerate}
\end{lem}
\begin{proof}
We will use Lemma~\ref{loxo condition} to prove this result.

    First, we prove that the element $R_1R_2R_3$ is loxodromic when $p_1=8$ with $p_2\ge 10$, or when $p_1=9$ with $p_2\ge 23$. Direct computation yields
\begin{eqnarray*}
    f_1(w,x,x,x) &=&(3-w)\bigl(x^2-(6-w)x+6\bigr)-2,\\
    F(w, x, x, x)&=&(3 - w)(x - w)^2 \left((3 - w) x^2 - 4\right).
\end{eqnarray*}
We see that $f_1(w, x, x, x)\ge 0$ if $p_1=8$ with $p_2 \ge 10$, or if $p_1=9$ with $p_2 \ge 18$, and that $F(w, x, x, x)>0$ if $p_1=8$ with $p_2 \ge 10$, or if $p_1=9$ with $p_2 \ge 23$. This confirms the claim by Lemma~\ref{loxo condition}.

Second, we consider the case $p_1=8$ with $p_2=8$ or $9$. In these cases, we can compute
\begin{enumerate}
    \item[(a)] $f_1(w,w,w,z)>0$ and $F(w,w,w,z)>0$ for $p_1=p_2=p_3=8$ and $p_4=10$; \medskip
    \item[(b)] $f_1(w,x,y,y)>0$ and $F(w,x,y,y)>0$ for $p_4=p_3=p_2+1$ with $p_2\in\{8,9\}$;\medskip
    \item[(c)] $f_1(w,x,x,z)>0$ and $F(w,x,x,z)>0$ for $p_2=p_3=9$ and $p_4=10$.
\end{enumerate}

Then, we can check by inspection that $R_1R_2R_3$ is elliptic in $\Gamma(8, 8, 8; 8)$, $\Gamma(8, 8, 8; 9)$ and $\Gamma(8,9,9,9)$.

Finally, we consider the case $p_1=9$ with $9\le p_2\le 22$. By direct computation, we have 
\begin{enumerate}
    \item[(a)] $f_1(w, x, x, z)>0$ and $F(w, x, x, z)>0$ for $p_4=p_2+\delta(p_2)+1$;\medskip
    \item[(b)] $f_1(w, x, y, y)>0$ and $F(w, x, y, y)>0$ for $p_4=p_3=p_2+2$ with $9 \le p_3 \le 16$;\medskip
    \item[(c)] $f_1(w, x, y, y)>0$ and $F(w, x, y, y)>0$ for $p_4=p_3=p_2+1$ with $17 \le p_2 \le 22$.
\end{enumerate}
This implies that, by Lemma~\ref{loxo condition}, except for the groups listed in (ii) and (iii), the element $R_1R_2R_3$ is loxodromic in all $\Gamma(p_1, p_2, p_3; p_4)$ groups with $p_1=9$. We can also check that $F(w,x,y,z)<0$ for all the 47 groups listed in (ii) and (iii).
\end{proof}

\medskip

When $p_1\ge 10$ there are only finitely many $p_2\ge p_1$ so that $w+x\le 50/9$. 
In the following table, using Lemma~\ref{loxo condition}, we list the quadruples $(p_1,p_2,p_3',p_4')$ with $p_1\le p_2\le p'_3\le p'_4$ so that $w+x\le 50/9$ for which
$R_1R_2R_3$ is loxodromic in the group $\Gamma(p_1,p_2,p_3,p_4)$ for all $p_3\ge p_3'$ and $p_4\ge p_4'$. Note that when $p_1=10$ and $10\le p_2\le 14$ we can take $p'_3=p_2$.
\begin{table}[htbp]
\centering
$$
\begin{array}{| ll | l |}
\hline
p_1 & p_2 & \{(p_3',p_4')\} \\
\hline
10 & 10 & (10,21),\,(11,16),\,(12,14),\,(13,13) \\
10 & 11 & (11,25),\, (12,19),\, (13,16),\,(14,15),\,(15,15) \\
10 & 12 & (12,30),\,(13,22),\,(14,19),\,(15,17),\,(16,16) \\
10 & 13 & (13,39),\,(14,26),\,(15,22),\,(16,20),\,(17,18),\,(18,18) \\
10 & 14 & (14,62),\,(15,33),\,(16,26),\,(17,23),\,(18,21),\,(19,20),\,(20,20) \\
10 & 15 & (16,46),\,(17,33),\,(18,28),\,(19,25),\,(20,23),\,(21,22),\,(22,22)\\
10 & 16 & (17,89),\,(18,45),\,(19,35),\,(20,31),\,(21,28),\,(22,26),\,(23,25),\,(24,24) \\
10 & 17 & (19,83),\,(20,49),\,(21,40),\,(22,35),\,(23,32),\,(24,30),\,(25,28),\,(26,27),\,(27,27)\\
10 & 18 & (21,110),\,(22,60),\,(23,47),\,(24,41),\,(25,37),\,(26,34),\,(27,32),\,(28,31)\\
& & (29,30),\,(30,30) \\
10 & 19 & (24,86),\,(25,61),\,(26,51),\,(27,45),\,(28,41),\,(29,39),\,(30,37),\,(31,35),\\
& & (32,34),\,(33,33)\\
10 & 20 & (27,107),\,(28,74),\,(29,61),\,(30,54),\,(31,49),\,(32,46),\,(33,43),\,(34,41),\\
& & (35,40),\,(36,39),\,(37,37)\\
10 & 21 & (31,126),\,(32,89),\,(33,74),\,(34,65),\,(35,59),\,(36,55),\,(37,52),\,(38,50),\\
& & (39,48),\,(40,46),\,(41,45),\,(42,44),\,(43,43)\\
10 & 22 & (36,205),\,(37,124),\,(38,99),\,(39,86),\,(40,77),\,(41,71),\,(42,67),\,(43,63),\\
& & (44,60),\,(45,58),\,(46,56),\,(47,54),\,(48,53),\,(49,52),\,(50,51),\,(51,51) \\
10 & 23 & (44,332),\,(45,178),\,(46,138),\,(47,117),\,(48,105),\,(49,96),\,(50,90),\\
& & (51,85),\,(52,81),\,(53,77),\,(54,74),\,(55,72),\,(56,70),\,(57,68),\,(58,67),\\
& & (59,65),\,(60,64),\,(61,63),\,(62,62)\\
10 & 24 & (59,923),\,(60,306),\,(61,226),\,(62,188),\,(63,166),\,(64,151),\,(65,139),\\
& & (66,131),\,(67,124),\,(68,118),\,(69,113),\,(70,109),\,(71,106),\,(72,103),\\
& & (73,100),\,(74,98),\ldots,(84,84)\\
10 & 25 & (113, 1420),\ldots,(160,160)\\
\hline
11 & 11 & (12,42),\,(13,25),\,(14,21),\,(15,18),\,(16,17),\,(17,17)\\
11 & 12 & (14,45),\,(15,30),\,(16,25),\,(17,22),\,(18,20),\,(19,19) \\
11 & 13 & (16,72),\,(17,40),\,(18,32),\,(19,28),\,(20,25),\,(21,24),\,(22,23),\,(23,23)\\
11 & 14 & (19,91),\,(20,51),\,(21,40),\,(22,35),\,(23,32),\,(24,30),\,(25,28),\,(26,27),\,(27,27)\\
11 & 15 & (23,241),\,(24,77),\,(25,57),\,(26,49),\,(27,44),\,(28,40),\,(29,38),\,(30,36),\\
& & (31,34),\,(32,33),\,(33,33)\\
11 & 16 & (31,135),\,(32,92),\,(33,76),\,(34,66),\,(35,60),\,(36,56),\,(37,53),\,(38,50),\\
& & (39,48),\,(40,47),\,(41,45),\,(42,44),\,(43,43)\\
11 & 17 & (49,783),\ldots,(70,70)\\
\hline
12 & 12 & (17,66),\,(18,41),\,(19,33),\,(20,29),\,(21,27),\,(22,25),\,(23,24),\,(24,24)\\
12 & 13 & (22,95),\,(23,59),\,(24,48),\,(25,42),\,(26,38),\,(27,36),\,(28,34),\,(29,32),\\
& & (30,31),\,(31,31)\\
12 & 14 & (33,154),\,(34,103),\,(35,84),\,(36,73),\,(37,66),\,(38,62),\,(39,58),\,(40,55),\\
& & (41,53),\,(42,51),\,(43,49),\,(44,48),\,(45,47),\,(46,46)\\
\hline
13 & 13 & (40,185),\,(41,130),\,(42,107),\,(43,94),\,(44,86),\,(45,79),\,(46,75),\,
(47,71),\, \\
& & (48,68),\,(49,65),\,(50,63),\,(51,61),\,(52,60),\,(53,58),\,(54,57),\,(55,56),\,(56,56)
\\
\hline
\end{array}
$$
\caption{}
\label{tab:list}
\end{table}

\end{document}